\documentclass[11pt]{article}
\usepackage[utf8]{inputenc}
\usepackage{amssymb,graphicx,mathtools,amsfonts}
\usepackage{amsmath}
\usepackage{amsthm}
\usepackage{hyperref}
\usepackage[dvipsnames]{xcolor}
\usepackage{dsfont}

\usepackage{a4wide}
\usepackage{esint}
\catcode`\@=11
\def\theequation{\@arabic{\c@section}.\@arabic{\c@equation}}
\catcode`\@=12
\graphicspath{{images/}}
\newtheorem{theorem}{Theorem}[section] %The theorems will be numbered section-wise
\newtheorem{lemma}{Lemma}[section] %The lemmas will be numbered section-wise
\newtheorem{proposition}{Proposition}[section]

\newtheorem*{remark}{Remark}

\newcommand{\e}{\epsilon}

\newcommand{\al} {\alpha}

\newcommand{\ga} {\gamma}

\newcommand{\om} {\Omega}

\newcommand{\la} {\lambda}

\newcommand{\noi} {\noindent}
\newcommand{\na} {\nabla}

\newcommand{\ds} {\displaystyle}
\newcommand{\ra} {\rightarrow}
\newcommand{\real}{\mathbb{R}}

\newcommand{\rnn}{\mathbb{R}^{N}}
\newcommand{\lv}{\lVert}
\newcommand{\rv}{\rVert}
\newcommand{\weak}{\rightharpoonup}

\newcommand{\sobo}{W_0^{1,p}}

\newcommand{\grad}{\nabla}
\newcommand{\ntrl}{\mathbb{N}}

\newcommand{\plap}{\Delta_p}
\newcommand{\pfrac}{(-\Delta_p)^s}
\newcommand{\rnnn}{\int_{\mathbb{R}^N}\int_{\mathbb{R}^N}}
\newcommand{\sobop}{W_0^{1,p}}
\newcommand{\fucik}{Fu\v{c}\'{\i}k } % for Fucik name

\title{On the eigenvalues and Fu\v{c}\'{\i}k spectrum of $p$-Laplace local and nonlocal operator with mixed interpolated Hardy term}
\author{Shammi Malhotra\footnote{Department of Mathematics, Indian Institute of Technology Delhi, Hauz Khas New Delhi 110016,  India, maz228085@maths.iitd.ac.in}, Sarika Goyal\footnote{Department of Mathematics, Netaji Subhas University of Technology,
		Dwarka Sector-3, Dwarka, Delhi, 110078, India, sarika1.iitd@gmail.com, sarika@nsut.ac.in},  and K. Sreenadh\footnote{Department of Mathematics, Indian Institute of Technology Delhi, Hauz Khas New Delhi 110016,  India, sreenadh@maths.iitd.ac.in}}
\date{}

\begin{document}
\maketitle
%%%%%%%%%%%%%%%%%%%%%%%%%%%%%%%%%%%%%%%%%%%%%%%%%%%%%%%%
%%%%%%%%%   Introduction
\begin{abstract}
\noi In this article, we are concerned with the eigenvalue problem driven by the mixed local and nonlocal $p$-Laplacian operator having the  interpolated Hardy term
\begin{equation*}
\mathcal{T}(u) :=- \Delta_p u + (- \Delta_p)^s u - \mu \frac{|u|^{p-2}u}{|x|^{p \theta}} 
%= \la |u|^{p-2}u  \mbox{ in } \om, \quad u =0 \mbox{ in } \rnn \setminus \Omega,
\end{equation*}
where $0<s<1<p<N$, $\theta \in [s,1]$, and $\mu \in (0,\mu_0(\theta))$. First, we establish a mixed interpolated Hardy inequality and then show the existence of eigenvalues and their properties. We also investigate the Fu\v{c}\'{\i}k spectrum, the existence of the first nontrivial curve in the Fu\v{c}\'{\i}k spectrum, and prove some of its properties.
Moreover, we study the shape optimization of the domain with respect to the first two eigenvalues, the regularity of the eigenfunctions, the Faber-Krahn inequality, and a variational characterization of the second eigenvalue. 
\medskip

\end{abstract}

\section{Introduction}
In this paper, we are interested in the eigenvalue problem of the mixed local-nonlocal operator with an interpolated Hardy potential term. More precisely, we are concerned with all $\la \in \real$ such that the following equation has a nontrivial solution $u$ 
\begin{equation}\label{spectrum_equation}
\mathcal{T}(u)  = \lambda u^{p-1}  \mbox{ on } \om \mbox{,  }
    u  = 0  \mbox{ on } \rnn\setminus\om,    
\end{equation}
where $\mathcal{T}(u) := - \plap u + \pfrac u - \mu \frac{|u|^{p-2}u}{|x|^{p \theta}}$, $0<s<1<p<N$, $\plap u:= div(|\grad u|^{p-2} \grad u ) $ and fractional $p$-Laplacian is defined as $$ \pfrac u(x) := - 2 \lim_{\e \to 0} \int_{\rnn \setminus B_{\e}(x)} \frac{|u(x) - u(y)|^{p-2}(u(x) - u(y))}{|x-y|^{N+ps}} ~dy $$
for all $x \in \rnn$ and $\theta \in [s,1]$,  $\mu \in (0,\mu_0(\theta))$ with
\begin{equation*}
\mu_0(\theta) = 
\begin{cases}
    C_{N,p,s} &  \mbox{ if } \theta  =s, \\
    \min\left\{ \frac{C_H(1-s)}{\theta - s}, \frac{C_{N,p,s}(1-s)}{1 - \theta} \right\} & \mbox{ if } \theta \in (s,1),\\
    C_H & \mbox{ if } \theta = 1, 
\end{cases}
\end{equation*}
where $C_H$ and $C_{N,p,s}$ are defined in \eqref{hardy_inequality} and \eqref{fractional_hardy_inequality} respectively. \\
% If $\alpha = \beta$, then the \fucik spectrum becomes the usual spectrum of the following problem:
% \begin{equation}\label{spectrum_equation}
% \mathcal{T}(u)  = \lambda u^{p-1}  \mbox{ on } \om \mbox{,  }
%     u  = 0  \mbox{ on } \rnn\setminus\om.    
% \end{equation}
% This compels us to begin our study with an analysis of the eigenvalue problem associated with the operator $\mathcal{T}$. 
% One of the first steps in understanding an operator is analyzing its spectrum.
Once the eigenvalue problem and its properties are studied, we further explore the first nontrivial curve in the \fucik spectrum. The \fucik spectrum  is defined as the set of all $(\al,\beta) \in \real^2$ such that there exists a nontrivial solution to the following problem 
\begin{equation}\label{pq}
\begin{cases}
      \mathcal{T}(u)  = \alpha (u^+)^{p-1} - \beta (u^-)^{p-1} & \mbox{ in }  \om  \\
    u  = 0  &\mbox{ in } \rnn \setminus \om ,
\end{cases}
\end{equation}

\noi  This concept was introduced by \fucik and Dancer while studying problems with jumping nonlinearities in one dimension (see \cite{fucik_book_introduction}). The higher-dimensional counterpart of the first nontrivial curve of the \fucik spectrum of the Laplacian operator with Dirichlet boundary conditions was investigated in \cite{figueiredo_first_non_trivial_curve}. Subsequently, the \fucik spectrum for the $p$-Laplacian operator with Dirichlet boundary conditions was studied in \cite{Cuesta1999fucik}. This line of research has since expanded to include the \fucik spectrum of both Laplacian and $p$-Laplacian operators under Neumann and Robin boundary conditions (see \cite{patrick_different_boundary, greta_robin_eigenvalues} and references therein). Additionally, eigenvalue problems involving singular weights in the equation, as well as those incorporating a Kirchhoff-type weight with the operator, have been explored (see \cite{ahmed_singular_weights, shibata_Kirchhoff}). Moreover, limiting eigenvalue problems as $p \to \infty$ have also been studied (see \cite{amato_limiting_p_eigenvalue}). The eigenvalue problem for the nonlocal $p$-Laplacian operator is first studied for $p \geq 2$ in \cite{lindqvist_fractional_eigenvalues_pgeq2}, and eventually generalized for all $p > 1$ in \cite{franzina_fractional_p_ev}. The \fucik spectrum of the fractional $p$-Laplacian was subsequently explored in \cite{sarika_fucik_fractional_laplacian, kanishka_fucik_fractional_p_laplacian}.\\

% Firstly, we study the existence of the eigenvalues of the operator $\mathcal{T}$. Then we dwelve deeply into the analysis of the first two eigenvalues and prove some properties of them as well as their variational characterization. 
\noindent In the recent years, significant attention has been given to the study of the combination of local and nonlocal operators. Both local and nonlocal factors influence many physical phenomena such as instrumental in describing diverse processes like diffusion of biological populations in ecological niches (see \cite{ valdinoci_logistic_equation, valdinoci_ecological_niche, pellacci_logistic_equation}), and anisotropic heat transport in reversed shear magnetic fields \cite{blazevski_anisotropic_heat_transport}. It also helps in investigation of optimal dispersal strategies in spatially heterogeneous environments \cite{pellacci_best_dispersal_strategies}. Consequently, understanding these operators has become a matter of great importance. Several results have been obtained for combinations of local and nonlocal operators. The first and second eigenvalues of such operators were studied in \cite{pezzo_eigenvalue_mixed_p, divya_second_eigenvalue_mixed}, respectively. In the special case where $\mu = 0$, the work in \cite{Bingzhong_yang_mixed_eigenvalue} established the existence of an unbounded sequence of decreasing curves in the spectrum.\\

\noindent Many researchers have started to observe the spectrum with an added potential to the operator with singularity in order to gain a better understanding of the spectrum of fundamental operators. The Hardy potential is one of the fundamental potentials that is incorporated and is driven by the following Hardy's inequality,
\begin{equation}\label{hardy_inequality}
    C_H \int_{\om} \frac{|u|^p}{|x|^p}~ dx \leq \int_{\om} |\grad u |^p ~ dx
\end{equation}
for all $1< p < N$ and for all $u \in \sobo(\om)$ where $C_H = \left( \frac{N-p}{p}\right)^p$. Sreenadh in \cite{ sreenadh_fucik_hardy, sreenadh_second_ev_hardy} investigated the following problem 
\begin{align*}
     L_{\mu}(u) : = - \plap u - \frac{\mu}{|x|^p} |u|^{p-2}u & = \al V(x) (u^+)^{p-1} - \beta V(x) (u^-)^{p-1}   \mbox{   in  } \om \\
     u = 0 \mbox{  on  }  \partial \om
\end{align*}
and proved the existence of the first, second eigenvalue and the first nontrivial curve in the \fucik spectrum of this Hardy-Sobolev operator firstly for weight $V(x) \equiv 1$ and then later for more general conditions on weight $V(x)$. Finally, in \cite{sarika_fucik_pfractional}, the author extended the study to the fractional $p$-Laplacian operator with Hardy potential, and a weight $V(x)$. 
%and also nonresonance problem with respect to the weighted \fucik spectrum is solved. 
In the nonlocal setting, the Hardy potential is motivated and handled by the following nonlocal version of Hardy inequality \cite{frank_fractional_Hardy},
\begin{equation}\label{fractional_hardy_inequality}
    C_{N,p,s} \int_{\rnn} \frac{|u|^p}{|x|^{ps}} dx \leq \int_{\rnn} \int_{\rnn} \frac{|u(x) - u(y)|^p}{|x-y|^{N+sp}} dy \, dx \mbox{ for all } u \in W_s^p(\rnn),
\end{equation} 
where $N \geq 1$, $s \in (0,1)$, $1 < p < N/s,$
$$C_{N,p,s} = 2 \int_0^1 r^{ps-1}|1-r^{(n-ps)/p}|^p \Psi_{N,s,p}(r) ~ dr,$$ and 
\begin{equation*}
    \Psi_{N,s,p}(r) := 
    \begin{cases}
            |\mathbb{S}^{N-2}| \int_{-1}^1 \frac{(1-t^2)^{(n-3)/2} ~ dt}{(1-2rt+r^2)^{(N+ps)/2}}, & N \geq 2, \\
            \left( \frac{1}{(1-r)^{1+ ps}} + \frac{1}{(1+r)^{1+ps}} \right), & N =1. 
    \end{cases}
\end{equation*}    
\noi Inspired by all the above works, our primary objective is to investigate the spectrum, \fucik spectrum, regularity results and shape optimization problem of the mixed operator $\mathcal{T}$ in the presence of a mixed interpolated Hardy potential. The novelty of this article is that we consider the mixed interpolated Hardy potential, which, to the best of the authors' knowledge, has not been previously studied. This potential is inspired by the interplay between classical and nonlocal Hardy inequalities, which lead to the mixed interpolated Hardy inequality that is proved in section $2$. The addition of the mixed interpolated Hardy potential introduces several challenges. Firstly, the pointwise convergence of the gradient is required for the Palais-Smale sequences, which is established by modifying the arguments in Lemma $2.2$ of \cite{silva_mixed}. Secondly, the presence of the singular Hardy potential affects the regularity of the eigenfunctions, requiring the use of local estimates. To address this, Moser iteration is adapted to prove the local boundedness of the eigenfunctions in Section $6$.
\\

\noindent We start our study of the spectrum by investigating the eigenvalues of the operator $\mathcal{T}$, particularly the first two eigenvalues. The existence of eigenvalues is proved using the index theory, particularly with the help of Krasnoselskii genus. The first eigenvalue can be characterized as 
\begin{equation}\label{first_eigenvalue}
    \lambda_1(\om) = \inf_{u \in X(\om)} \left\{ \int_{\om} |\grad u|^p + \int_{\rnn} \int_{\rnn} \frac{|u(x) - u(y)|^p}{|x-y|^{N+sp}} -\mu \int_{\om} \frac{|u|^p}{|x|^{p\theta}} : \int_{\om} |u|^p = 1 \right\}
\end{equation}
where $X(\om)$ is defined in \eqref{function_space}, and it's proved to be isolated, simple. Meanwhile, eigenfunctions corresponding to other eigenvalues change sign. \\
% \textcolor{red}{this is the first time this mixed operator with Hardy is explored.}

\noindent After this, the \fucik spectrum is studied, and the first nontrivial curve is constructed. We end the study by answering the shape optimization problems of the domain for the first two eigenvalues. The shape optimization problem for the first eigenvalue is answered by proving the Faber-Krahn inequality, which states the ball has the smallest eigenvalue in the set of all domains with fixed volume ``$c$". For the second eigenvalue shape optimization problem, which states the smallest second eigenvalue is obtained for the disjoint union of two balls of equal volume ``$c/2$" among the class of all domains with fixed volume ``$c$", is proved with help of the mixed Hong-Krahn-Szeg\"{o} inequality. \\

\noi The structure of this paper is organized as follows. Section $2$ introduces the necessary notations and state fundamental lemmas and propositions. In Section $3$, we establish the mixed interpolated Hardy inequality and prove a key lemma concerning the pointwise convergence of the gradient for a bounded (P.-S.) sequence. The section concludes with a proof of the (P.-S.) condition for the associated functional $\mathcal{J}_d$ on the set $S$, which will be defined later. In section $4$, we study the eigenvalue problem associated with the operator $\mathcal{T}$ and prove some properties of eigenvalues and eigenfunctions.
% of first eigenvalue like simplicity, positivity and isolation in the spectrum. Additionally, we show that higher-order eigenfunctions change sign. 
Section $5$ introduces the \fucik spectrum for the operator $\mathcal{T}$, and the first nontrivial curve in the \fucik Spectrum is constructed, and some of its basic properties are obtained. Finally, Section $6$ explores shape optimization problems with respect to the first two eigenvalues for which we prove the Faber-Krahn inequality, regularity of eigenfunctions and the mixed Hong-Krahn-Szeg\"{o} inequality.

\section{Preliminaries}

We start this section by introducing some notations and spaces that are used in the paper and also state some basic lemmas and propositions that will be required. 
To study the problem, we will work with the following space
\begin{equation}\label{function_space}
X(\om) := \left\{ u \in W^{1,p}(\rnn) : u|_\om  \in W_{0}^{1,p}(\om) \mbox{ and } u = 0  \mbox{ a.e. in } \rnn\setminus\om \right\},
\end{equation}
where $W^{1,p}(\rnn) = \{ u \in L^p(\rnn) : \grad u \in  (L^p(\rnn))^N \}$ is the Sobolev space with the norm $\lv u \rv_{W^{1,p}} = \left( \int_{\om} |\grad u |^p + \int_{\om} |u|^p \right)^{1/p}$ and $W_0^{1,p}(\om)$ is the closure of the $C_c^{\infty}(\om)$ functions with respect to the norm $\left( \int_{\om} |\grad u|^p \right)^{\frac{1}{p}}$. $X(\om)$ is endowed with the following norm $\lv u \rv = \left( \int_{\om} |\grad u |^p + [u]_{s,p}^p \right)^{1/p},$ where $[u]_{s,p}^p = \int_{\rnn} \int_{\rnn} \frac{|u(x) - u(y)|^p}{|x-y|^{N + ps}} \, dx ~ dy$.\\
Moreover, we take 
\begin{equation*}
    \begin{split}
    C^{1,1}(\om) := \{ &u: \om \to \real : \text{there exists } C > 0 \text{ such that } \\
    & |Du(x) - Du(y)| \leq C |x - y| \text{ for all } x,y \in \om \}.
    \end{split}
\end{equation*}
\noi Now, for $1 < p < \infty$, we consider the function $F : \real \to \real$ defined as
\begin{equation*}
F(t) = 
\begin{cases}
    |t|^{p-2} t & \mbox{ if } t \neq 0,\\
    0 & \mbox{ if } t = 0 .
\end{cases}
\end{equation*}
We state the following useful propositions involving this function

\begin{proposition}[Lemma $C.2$ of \cite{brasco_cheeger}]
Let $p \in (1, \infty)$ and $\gamma \geq 1$, then for every $a,b,M \geq 0,$ we have 
\begin{equation}\label{fractional_minimum_inequality}
    F(a-b)(a_M^{\gamma} - b_M^{\gamma}) \geq \frac{\ga p^p}{(\ga + p -1)^p} \left| a_M^{\frac{\ga + p -1}{p}} - b_M^{\frac{\ga + p -1}{p}} \right|^p,
\end{equation}
where $a_M = \min \{a,M\}$ and $b_M = \min\{ b, M\}$.
\end{proposition}

\begin{proposition}[Formula $2.2$ of \cite{simon_inequalities}]
For all $t_1, t_2 \in \rnn$, there exists a constant $C>0$ such that the following holds 
\begin{equation}\label{simon_inequality}
    \langle F(t_1) - F(t_2), t_1 - t_2 \rangle \geq 
\begin{cases}
    C | t_1 - t_2|^p & \mbox{ if } p \geq 2, \\
    C \frac{|t_1 - t_2|^2}{( |t_1| + |t_2|)^{2-p}} & \mbox{ if } 1 < p \leq 2.
\end{cases}
\end{equation}
\end{proposition}
\noi It is quite useful in getting strong and pointwise convergence of sequences. 
\begin{proposition}[Lemma $A.2$ of \cite{brasco_second_eigenvalue}]
Let $1 < p < \infty$ and let $g : \real \to \real$ be an increasing function. Take $$ G(t) = \int_{0}^t (g^{\prime}(\tau))^{\frac{1}{p}} \, d \tau.$$ Then we have 
\begin{equation}\label{brasco_inequality}
    F(a-b)(g(a) - g(b)) \geq |G(a) - G(b)|^p.
\end{equation}
\end{proposition}
\noi   Taking $g(t)$ equal to $t^+$ and $-t^-$ in \eqref{brasco_inequality}, we obtain the following inequalities
\begin{equation}\label{brasco_inequality_positive}
    F(a-b)(a^+ - b^+) \geq |a^+ - b^+|^p ,
\end{equation}
and
\begin{equation}\label{brasco_inequality_negative}
    -F(a-b)(a^- - b^-) \geq |a^- - b^-|^p,
\end{equation}
for $a, b \in \real$ and $1 < p < \infty$.
% \begin{equation}\label{hardy_inequality}
%     C_H \int_{\om} \frac{|u|^p}{|x|^p}~ dx \leq \int_{\om} |\grad u |^p ~ dx
% \end{equation}
% for all $1< p < N$ and for all $u \in \sobo(\om)$ where $C_H = \left( \frac{N-p}{p}\right)^p$. Also, with the fractional Hardy inequality is given by 
% \begin{lemma}
%     Let $N \geq 1$ and $s \in (0,1)$. Then for all $u \in W_p^s(\rnn)$ when $1 < p < N/s$ and for all $u \in W_p^s(\rnn \setminus \{ 0\} )$ when $p > N/s$, we have $$C_{N,p,s} \int_{\rnn} \frac{|u|^p}{|x|^{ps}} \leq \int_{\rnn} \int_{\rnn} \frac{|u(x) - u(y)|^p}{|x-y|^{N+sp}} $$ where $$C_{N,p,s} = 2 \int_0^1 r^{ps-1}|1-r^{(n-ps)/p}|^p \Psi_{N,s,p}(r) ~ dr,$$ and 
%     \begin{equation*}
%         \Psi_{N,s,p}(r) := 
%         \begin{cases}
%             |\mathbb{S}^{N-2}| \int_{-1}^1 \frac{(1-t^2)^{(n-3)/2 ~ dt}}{(1-2rt+r^2)^{(N+ps)/2}}, & N \geq 2, \\
%             \left( \frac{1}{(1-r)^{1+ ps}} + \frac{1}{(1+r)^{1+ps}} \right), & N =1. 
%         \end{cases}
%     \end{equation*}    
% \end{lemma}
% Thus, we have the following inequality in view of the above two inequalities and Young's Inequality 

\noi To show the existence of eigenvalues of the operator $\mathcal{T}$, we would like to use the Theorem $5.7$ of \cite{struwe_variational_methods} stated as 
\begin{theorem}\label{thm:existence_ev_struwe}
Suppose $\mathcal{J} \in C^1(M)$ is an even functional on a complete symmetric $C^{1,1}$-manifold $M \subset X \setminus \{0\}$ in some Banach space $X$. Also suppose $\mathcal{J}$ satisfies $(P.\text{-}S.)$ and is bounded from below on $M$. Let
\[
\hat{\gamma}(M) = \sup \{ \gamma^*(K) : K \subset M \text{ compact and symmetric} \},
\]
where $\gamma^*$ is defined later in \eqref{genus}. Then the functional $\mathcal{J}$ possesses at least $\hat{\gamma}(M) \leq \infty$ pairs of critical points. Moreover, for any $k \leq \gamma^*(M)$, the values $\beta_k$(provided they are finite) given by 
$$\beta_k = \inf_{K \subset \mathcal{F}_k } \sup_{u \in K} E(u)$$ are critical points of $\mathcal{J}$, where $\mathcal{F}_k = \{ K \in \mathcal{A} : K \subset M, \gamma^*(K) \geq k\}.$

\end{theorem}
\noi Next, we recall the following propositions that will be helpful in proving the simplicity of the first eigenvalue of the local and nonlocal operator with mixed interpolated Hardy.
\begin{proposition}[Discrete Picone Identity \cite{amghibech_discrete_picone}] \label{discrete_picone_identity}
Let $p \in (1,\infty)$. For $u,v : \om \subset \rnn \to \real$ such that $u \geq 0$ and $v > 0$, we have $$ L (u,v) \geq 0 \mbox{ in } \rnn \times \rnn, $$ where $$ L(u,v) = |u(x)-u(y)|^p - |v(x) - v(y)|^{p-2}(v(x) - v(y))\left( \frac{u(x)^p}{v(x)^{p-1}} - \frac{u(y)^p}{v(y)^{p-1}} \right).$$
The equality holds if and only if $u = k v$ a.e. for some constant $k$.
\end{proposition}

\begin{proposition}[Continuous Picone Identity \cite{huang_continous_picone_identity}]\label{continuous_picone_identity}
Let us define for $u, v \in C^1(\Omega)$ with $u \geq 0$ and $v >0$
\[
L_c(u, v) = |\nabla u|^p - (p-1) \frac{u^p}{v^p} |\nabla v|^p - p \frac{u^{p-1}}{v^{p-1}} |\nabla v|^{p-2} \nabla v \nabla u,
\]
and
\[
R_c(u, v) = |\nabla u|^p - |\nabla v|^{p-2} \nabla v  \nabla \left( \frac{u}{v^{p-1}} \right).
\]
Then $R_c(u, v) = L_c(u, v) \geq 0$ and equal to $0$ if and only if $u = k v$ for some constant $k$.    
\end{proposition}

\noi Finally, we end this section by stating the following variant of the mountain pass Lemma \cite{Cuesta1999fucik}. 
\begin{lemma}[Mountain Pass Lemma]\label{mountain_pass_variant}
Let $X$ be a real Banach space and $M = \{ u \in X : g(u) = 1 \}$, where $g \in C^1(X;\real)$. Suppose $\mathcal{J} \in C^1(X;\real), u_0,u_1 \in M$ and let $\e > 0$ be such that $ \lv u_1 - u_0 \rv > \e$ and $$ \inf\{\mathcal{J}(u): u \in M \mbox{ and } \lv u - u_0 \rv_X = \e \} > max\{ \mathcal{J}(u_0), \mathcal{J}(u_1)\},$$ then there exists a sequence $\{u_k\} \in M$ such that $\mathcal{J}(u_k) \to c$ and $\lv \bar{\mathcal{J}}^{\prime}(u_k) \rv_* \to 0 $, where 
\begin{equation}\label{derivative_on_manifold}
    \lv \bar{\mathcal{J}}^{\prime}(u) \rv_* = \inf \{ \lv {\mathcal{J}}^{\prime}(u) - t g^{\prime}(u) \rv_{X^*} : t \in \real \}
\end{equation}
and $c = \inf_{\gamma \in \Gamma} sup_{\gamma} \mathcal{J}(u)$ with $\Gamma = \{ \gamma \in C([-1,1],M): \gamma(-1) = u_0, \gamma(1) = u_1 \}$. Moreover, if $\mathcal{J}$ satisfies the (P.-S.) condition on $M$, then there exists a $u \in M$ such that $\mathcal{J}(u) = c$ and $\lv \bar{\mathcal{J}}^{\prime}(u) \rv_* = 0$.
\end{lemma}

%%%%%%%%%%%%%%%%%%%%%%%%%%%%%%%%
%%%% Dedicated section
%%%%%%%%%%%%%%%%%%%%%%%%%%%%%%%
\section{Interpolated Hardy Inequality and Convergence results}

%%%%%%%%%% 
%%%% mixed interpolated Hardy inequality
%%%%%%%%%%

\noi First, we state and prove the mixed interpolated Hardy inequality, which is inspired by both the classical and fractional Hardy inequalities. Next, we establish a key lemma that ensures the pointwise convergence of the gradients of the (P.-S.) sequence, facilitating the passage to the limit. Finally, we conclude the section by proving the (P.-S.) condition for the associated functional $\mathcal{J}_d$ on the set $S$.

\begin{lemma}[Mixed Interpolated Hardy Inequality]
    Let $\theta \in [s,1]$ with $s \in (0,1)$ and  $1 < p < N$. Then for all $u \in X(\om)$, the following holds 
\begin{equation}\label{interpolated_hardy_with_theta}
   \int_{\om} \frac{|u|^{p}}{|x|^{p \theta}} dx \leq \frac{(\theta -s)}{(1-s)C_H} \int_{\om} |\grad u |^p ~ dx + \frac{(1-\theta)}{(1-s)C_{N,p,s}}\int_{\rnn} \int_{\rnn} \frac{|u(x) - u(y)|^p}{|x-y|^{N+sp}} dy\,dx.
\end{equation}
\end{lemma}
\begin{proof}
If $\theta = s$ and $\theta = 1$, then the inequality \eqref{interpolated_hardy_with_theta} follows from the classical Hardy inequality \eqref{hardy_inequality} and the fractional Hardy inequality \eqref{fractional_hardy_inequality}, respectively. \\
If $\theta \in (s,1)$, then we can write $\theta = (1 - \eta) s + \eta$ with $\eta = \frac{\theta - s}{1 - s} \in (0,1)$. Using this and H\"{o}lder's inequality and Young's inequality with exponents $1/\eta$ and $1/(1-\eta)$, we obtain
\begin{align*}
\int_{\om} \frac{|u|^{p}}{|x|^{p \theta}}  &= \int_{\om} \frac{|u|^{p(1 - \eta + \eta)}}{|x|^{p ((1-\eta) s+ \eta)}} \leq \left( \int_{\om} \frac{|u|^p}{|x|^{ps}}\right)^{1 - \eta} \left( \int_{\om} \frac{|u|^p}{|x|^p} \right)^\eta  \leq ( 1- \eta) \int_{\om} \frac{|u|^p}{|x|^{ps}} + \eta \int_{\om} \frac{|u|^p}{|x|^p}.
\end{align*}
Finally, the Hardy inequality \eqref{hardy_inequality} and the fractional Hardy inequality \eqref{fractional_hardy_inequality} conclude our result.
\end{proof}

\begin{remark}
Taking $\xi = \frac{\theta - s}{1 - s}$ in \eqref{interpolated_hardy_with_theta}, we have
\begin{equation}\label{interpolated_hardy}
   \int_{\om} \frac{|u|^{p}}{|x|^{p \theta}} dx \leq \frac{\xi}{C_H} \int_{\om} |\grad u |^p ~ dx + \frac{(1-\xi)}{C_{N,p,s}}\int_{\rnn} \int_{\rnn} \frac{|u(x) - u(y)|^p}{|x-y|^{N+sp}}dy \, dx.
\end{equation}
\end{remark}

\noi Fix $d>0$, consider the functional $\mathcal{J}_d: X(\om) \to \mathbb{R}$ as
\begin{equation}\label{energy_functional_d}
\mathcal{J}_d(u) :=  \int_{\om} |\grad u|^p \, dx + \rnnn \frac{|u(x)-u(y)|^{p}}{|x-y|^{N+ps}} \, dx \, dy - \mu \int_{\Omega} \frac{|u|^p}{|x|^{p\theta}} \, dx - d \int_{\Omega} (u^+)^p \, dx.
\end{equation}
Then $\mathcal{J}_d \in C^1(X(\om), \mathbb{R})$ and for any $ v \in X(\om)$
\begin{align*}
\frac{1}{p}\langle \mathcal{J}_d'(u), v \rangle =& \int_{\om} |\grad u|^{p-2} \grad u \grad v + \rnnn \frac{|u(x)-u(y)|^{p-2}(u(x)-u(y))(v(x)-v(y))}{|x-y|^{N+p s}} \, dx \, dy \\
 & - \mu \int_{\Omega} \frac{|u|^{p-2} u v}{|x|^{p\theta}} \, dx - d \int_{\Omega} (u^+)^{p-1} v \, dx.
\end{align*}
Also, $\tilde{\mathcal{J}}_d := \mathcal{J}_d|_S \in C^1(X(\om), \mathbb{R})$, where $S$ is defined as
\noi Also, define 
\begin{equation}\label{unit_sphere}
    S := \left\{ u \in X(\om) : I(u) := \int_{\om} |u|^p = 1 \right\}.
\end{equation}
Moreover, we take $\mathcal{J} = \mathcal{J}_0$, i.e., $\mathcal{J} : X(\om) \to \real$ given as
\begin{equation}\label{energy_functional}
    \mathcal{J}(u) = \int_{\om} |\grad u|^p dx + \int_{\rnn} \int_{\rnn} \frac{|u(x) - u(y)|^p}{|x-y|^{N+sp}} dy~dx -\mu \int_{\om} \frac{|u|^p}{|x|^{p\theta}}dx.
\end{equation}

%%%%%%%%%%%%%%%%%%%%%%%%%%%%%%%%%%%%%%%%%%%%%%%%
%%%%%%%%%%%%%   pointwise convergence of the gradient
\noi As the gradient term appears in the Palais-Smale sequence, we require its pointwise convergence to pass to the limit. To establish this, we present the following lemma, which is inspired by Lemma $2.2$ of \cite{silva_mixed}.
\begin{lemma}\label{pointwise_gradient}
Let $\mathcal{J}_d$ be the functional defined as in \eqref{energy_functional_d} and $\tilde{\mathcal{J}_d} := \mathcal{J}_d|_S$. If  $\mathcal{J}_d(u_k) \to c$ for some $c \in \real$ and $\lv \tilde{\mathcal{J}_d}^{\prime}(u_k)\rv_{*} \to 0$, defined as in \eqref{derivative_on_manifold}, for a sequence $\{ u_k \}$ in $X(\om)$, then, up to a subsequence, we have $\grad u_k(x) \to \grad u(x) $ a.e. in $\om$ as $k \to \infty$.
\end{lemma}
\begin{proof}
Applying the mixed interpolated Hardy inequality, we have
\begin{align*}
    C + d \geq \tilde{\mathcal{J}_d}(u_k) + d \int_{\om} (u^+)^p \geq \left( 1 - \frac{\mu \xi}{C_H}\right) \int_{\om} |\grad u_k|^p + \left( 1 - \frac{\mu(1-\xi)}{C_{N,p,s}} \right)[u_k]_s^p.
\end{align*}
Then the sequence $\{ u_k \}$ is bounded in $X(\om)$. Therefore, up to a subsequence (still denoted by $u_k$), as $k \to \infty$, we have
\begin{align}\label{pointwise_cgs_weak_properties}
    &u_k \weak u \mbox{ weakly in } X(\om), \quad  \quad \quad \quad \quad \quad \grad u_k \weak \grad u \mbox{ weakly in } (L^p(\om))^N,& \notag \\
    & u_k(x) \to u(x)  \mbox{ pointwise a.e. in } \om,    \quad \quad |u_k(x)| \leq g(x) \mbox{ a.e. in } \om,& \\
&u_k \to u \mbox{ strongly in } L^r(\om),\notag &
\end{align}
where $r \in [p,p^*)$ and $g \in L^{p^*}(\om)$ with $p^* = \frac{pN}{N-p}$.\\
Additionally, as $\lv \tilde{\mathcal{J}_d}^{\prime}(u_k) \rv_{*} \to 0$, there exists a sequence $\{ t_k \}$ such that 
\begin{equation}\label{pointwise_cgs_derivative_functional}
    \left| \langle \mathcal{J}_d^{\prime}(u_k),v\rangle - t_k \int_{\om} |u_k|^{p-2} u_k v \right| \leq \e_k \lv v \rv
\end{equation}
for all $v \in X(\om)$. Taking $v = u_k$ in \eqref{pointwise_cgs_derivative_functional}, we deduce from the mixed interpolated Hardy inequality \eqref{interpolated_hardy_with_theta} that
\begin{align*}
    |t_k| & \leq \e_k \lv u_k \rv + \lv u_k \rv^p + \mu \int_{\om} \frac{|u_k|^p}{|x|^{p \theta}}  + d \int_{\om} (u_k^+)^{p} \leq \e_k \lv u_k \rv + C \lv u_k \rv^p + d.
\end{align*}
Thus, $\{ t_k \} $ is bounded. Now, consider the truncation functions $T_j : \real \to \real$ as 
\begin{equation*}
    T_j(t) =
    \begin{cases}
        t & \mbox{ if } |t| \leq j \\
        j\frac{t}{|t|} & \mbox{ if } |t| > j.
    \end{cases}
\end{equation*}
Then \eqref{pointwise_cgs_weak_properties} implies 
\begin{gather}\label{pointwise_cgs_eq1}
    \lim_{k \to \infty} \int_{\om} |\grad u |^{p-2} \grad u \grad( T_j( u_k - u)) ~ dx = 0 \\
    \label{pointwise_cgs_eq2}
    \lim_{k \to \infty} \rnnn \frac{F(u(x) - u(y)) (T_j(u_k - u)(x) - T_j(u_k -u)(y))}{|x - y |^{N+ps}}~dy ~dx = 0 \\
    \label{pointwise_cgs_eq0}
    \lim_{k \to \infty} \int_{\om} (u^+)^{p-1} T_j(u_k - u) ~ dx = 0
\end{gather}
Since for any measurable set $E \subset \om$, H\"{o}lder inequality and \eqref{pointwise_cgs_weak_properties} gives
$$ \left|\int_{E}  \frac{|u|^{p-2}uT_j(u_k - u)}{|x|^{p \theta}} \right| \leq \left(\int_{E} \frac{|u|^p}{|x|^{p \theta}}\right)^{\frac{p-1}{p}} \left( \int_{E} \frac{|T_j(u_k - u)|^p}{|x|^{p \theta}}\right)^p \leq C \left(\int_{E} \frac{|u|^p}{|x|^{p \theta}}\right)^{\frac{p-1}{p}}. $$
Thus, by Vitali convergence theorem, we deduce 
\begin{equation}\label{pointwise_cgs_eq3} 
    \lim_{k \to \infty} \int_{\om} \frac{|u|^{p-2}uT_j(u - u_k)}{|x|^{p \theta}} = 0.
\end{equation}
Using equations \eqref{pointwise_cgs_eq1},  \eqref{pointwise_cgs_eq2}, \eqref{pointwise_cgs_eq0}, and \eqref{pointwise_cgs_eq3} implies $\langle \mathcal{J}_d^{\prime}(u),T_j(u_k - u) \rangle = o_k(1)$. Thus, taking the test function $v = T_j( u_k - u )$ in \eqref{pointwise_cgs_derivative_functional} gives
\small{
\begin{equation}\label{eq4:pointwise_cgs}
    |\langle \mathcal{J}_d^{\prime}(u_k) - \mathcal{J}_d^{\prime}(u), T_j(u_k - u)\rangle | \leq t_k \left| \int_{\om} ( |u_k|^{p-2} u_k - |u|^{p-2}u)(T_j(u_k - u))  \right| + \e_k \lv T_j(u_k - u)\rv + o_k(1).
\end{equation}}
\noi From the inequality $(2.8)$ in \cite{silva_mixed}, and relation \eqref{eq4:pointwise_cgs}, we obtain
\begin{align*}
& \int_{\om} (|\grad u_k|^{p-2} \grad u_k - |\grad u|^{p-2}\grad u)\grad (T_j(  u_k -  u))\\
& \leq  \int_{\om} (|\grad u_k|^{p-2} \grad u_k - |\grad u|^{p-2}\grad u)\grad (T_j(  u_k -  u)) \\ 
& \quad + \rnnn \frac{[F(u_k(x) - u_k(y)) - F(u(x) - u(y))](T_j(u_k -u)(x) - T_j(u_k - u)(y))}{|x - y|^{N + p s}}  \\
& \leq  \mu \left| \int_{\om} \frac{(|u_k|^{p-2}u_k - |u|^{p-2}u)(T_j(u_k - u))}{|x|^{p \theta}}\right| + d \left| \int_{\om} ( (u_k^+)^{p-1} - (u^+)^{p-1})(T_j(u_k - u))  \right| \\
& \quad +   t_k \left| \int_{\om} ( |u_k|^{p-2} u_k - |u|^{p-2}u)(T_j(u_k - u)) \right| + \e_k \lv T_j(u_k - u)\rv + o_k(1).
\end{align*}
Using the strong convergence in \eqref{pointwise_cgs_weak_properties}, we deduce 
\begin{align*}
    \limsup_{k \to \infty} \int_{\om} (|\grad u_k|^{p-2} \grad u_k - & |\grad u|^{p-2}\grad u)(T_j( \grad u_k - \grad u))\\
    &\leq \mu \limsup_{k \to \infty} \left| \int_{\om} \frac{(|u_k|^{p-2}u_k - |u|^{p-2}u)(T_j(u_k - u))}{|x|^{p \theta}} \right| \\
    & = \mu \limsup_{k \to \infty} \left| \int_{\om} \frac{(|u_k|^{p-2}u_k)(T_j(u_k - u))}{|x|^{p \theta}} \right|\\
    & \leq \mu j \limsup_{k \to \infty} \left( \int_{\om} \frac{|u_k|^p}{|x|^{p \theta}}  \right)^{\frac{p-1}{p}} \left( \int_{\om} \frac{1}{|x|^{p \theta}} \right)^{\frac{1}{p}}\\
    & \leq  j C.
\end{align*}
Now, let $e_k(x) = [ |\grad u_k(x)|^{p-2} \grad u_k(x) - |\grad u(x)|^{p-2} \grad u(x) ]\grad (u_k(x) - u(x))$. Inequality \eqref{simon_inequality} imply that $e_k(x) \geq 0$. Consider the subsets of $\om$ as 
\begin{equation*}
    S_k^j = \{ x \in \om : |u_k(x) - u(x)| \leq j \}, \quad G_k^j = \{ x \in \om : |u_k(x) - u(x)| > j \}. 
\end{equation*}
Then for $\delta \in (0,1)$, we have 
\begin{align*}
\int_{\om} e_k^{\delta} &= \int_{S_k^j} e_k^{\delta} + \int_{G_k^j} e_k^{\delta}\\
& \leq \left( \int_{S_k^j} e_k \right)^{\delta} |S_k^j|^{1-\delta} + \left( \int_{G_k^j} e_k \right)^{\delta} |G_k^j|^{1-\delta} \\
& \leq (jC)^{\delta} |S_k^j|^{1-\delta} + (\bar{C})^{\delta} |G_k^j|^{1-\delta}.
\end{align*}
Since $|G_k^j| \to 0$ as $k \to \infty$, we obtain
$$0 \leq \limsup_{k \to \infty} \int_{\om} e_k^{\delta} \, dx \leq (j C)^{\delta} |\om|^{1-\delta}. $$
Taking $j \to 0^+$, we deduce that $e_k^{\delta} \to 0 $ in $L^{1} (\om)$. Subsequently, $e_k(x) \to 0$ a.e. in $\om$ as $k \to \infty$. Therefore, we have our result by \eqref{simon_inequality}.
\end{proof}

%%%%%%%%%%%%%%%
%%%% \mathcal{J}_d satisfies the PS condition
\noi Next, with the Lemma \ref{pointwise_gradient} in hand, we prove that the functional $\mathcal{J}_d$ satisfies the (P.-S.) condition on $S$.
\begin{lemma}\label{ps_condition}
$\tilde{\mathcal{J}}_d$ satisfies the (P.-S.) condition on $S$.
\end{lemma}
\begin{proof}
Let $\{ u_k \}$ be a Palais-Smale sequence of $\tilde{\mathcal{J}}_d$, then there exists $K > 0$ and a sequence $\{ t_k \}$ such that
\begin{equation}\label{ps_eq1}
 \mathcal{J}_d(u_k) \leq K    
\end{equation}
and
\begin{align}\label{ps_eq2}
&\left| \int_{\om} |\grad u_k|^{p-2} \grad u_k \grad v + \rnnn \frac{F(u_k(x) - u_k(y)) (v(x) - v(y))}{|x - y|^{N+ps}} \, dx \, dy  \right. \notag\\
&\qquad \left. -  \mu \int_{\Omega} \frac{|u_k|^{p-2}u_k v}{|x|^{p\theta}} \, dx - d\int_{\Omega} (u_k^+)^{p-2}u_k^+ v \, dx - t_k \int_{\Omega} |u_k|^{p-2}u_k v \, dx \right| \leq \epsilon_k \|v\|_{X(\om)},
\end{align} 
for all $v \in X(\om),$ and for some $ \epsilon_k > 0$ such that $\e_k \to 0$. From \eqref{ps_eq1}, using mixed interpolated Hardy inequality \eqref{interpolated_hardy} we deduce that $\{u_k\}$ is bounded in $X(\om)$. Indeed, we have 
$$ \left( 1 - \frac{\mu \xi}{C_H}\right) \lv u_k \rv^p \leq K + d  \int_{\om} (u_k^+)^p \leq K +d. $$
Therefore, there is a subsequence denoted by $\{u_k\}$ and $u \in X(\om)$ such that $u_k \rightharpoonup u$ weakly in $X(\om)$, and $u_k \to u$ strongly in $L^q(\Omega)$ for all $1 \leq q < pN/(N-p)$. Taking $v = u_k$ as a test function in \eqref{ps_eq2}, we have
\small{
\begin{align*}
|t_k| \leq & \left| \int_{\om} | \grad u_k |^p + \int_{\Omega} \int_{\Omega} \frac{|u_k(x) - u_k(y)|^{p}}{|x - y|^{N+ps}} \, dx \, dy  -  \int_{\Omega} \frac{|u_k|^{p-2} }{|x|^{p \theta}} \, dx  + d \int_{\Omega} (u_k^+)^{p}\, dx + \epsilon_k \|u_k \|_{X(\om)} \right| \leq C.
\end{align*}}

\noi Hence $\{t_k\}$ is a bounded sequence.
Next, we claim that $u_k \to u$ strongly in $X(\om)$. Taking $v = u_k - u$ as a test function in \eqref{ps_eq2}, we have
\begin{equation}\label{ps_eq3}
\begin{split}
& \left| \int_{\om} |\grad u_k|^{p-2} \grad u_k \grad (u_k - u)  + \rnnn \frac{F(u_k(x) - u_k(y)) ((u_k(x) - u(x)) - (u_k(y) - u(y)))}{|x - y|^{N+p s}} \right. \\
&\left. - \mu \int_{\Omega} \frac{|u_k|^{p-2}u_k (u_k - u)}{|x|^{p\theta}} \, dx \right| \leq O(\epsilon_k) + d \lv u_k^+ \rv_{p}^{p-1} \|u_k - u\|_{p}^p + |t_k| \|u_k\|_{p}^{p-1} \|u_k - u\|_{p} \to 0, \quad
\end{split}
\end{equation}
i.e., $|\langle \mathcal{J}^{\prime}(u_k), u_k -u\rangle| \to 0$ as $k \to \infty$. By Brezis-Lieb Lemma 2.2, we have
\begin{equation}\label{ps_eq4}
\begin{split}
\int_{\om} | \grad (u_k - u)|^p & = \int_{\om} |\grad u_k|^p - \int_{\om} |\grad u|^p + o_k(1)\\
[u_k -u]_s^p  & =  [u_k]_s^p - [u]_s^p + o_k(1)\\
\int_{\om} \frac{|u_k - u|^p}{|x|^{p \theta}} & =  \int_{\om} \frac{|u_k|^p}{|x|^{p \theta}} - \int_{\om} \frac{|u|^p}{|x|^{p \theta}} + o_k(1).
\end{split}
\end{equation}
\noi Let $p^{\prime} = \frac{p}{p-1}$ be the conjugate exponent of $p$. Then from Lemma \ref{pointwise_gradient} and $u_k(x) \to u(x)$ pointwise a.e. in $\om$ as $k \to \infty$, we deduce that 
\begin{align*}
    |\grad u_k(x)|^{p-2} \grad u_k(x) &\to |\grad u(x)|^{p-2} \grad u(x) &\mbox{ pointwise a.e. in } \om ,\\
    \frac{|u_k(x)|^{p-2}u_k(x)}{|x|^{\frac{p \theta}{p^{\prime}}}} &\to \frac{|u(x)|^{p-2}u(x)}{|x|^{\frac{p \theta}{p^{\prime}}}} &\mbox{ pointwise a.e. in } \om,  \\
    \frac{F(u_k(x) - u_k(y))}{|x-y|^{\frac{N+sp}{p^{\prime}}}} &\to \frac{F(u(x) - u(y))}{|x-y|^{\frac{N+sp}{p^{\prime}}}} &\mbox{ pointwise a.e. in } \mathbb{R}^{2N},
\end{align*}
as $k \to \infty$.  Moreover, $\{ |\grad u_k|^{p-2} \grad u_k\}$ and $\left\{ \frac{|u_k|^{p-2}u_k}{|x|^{\frac{p \theta}{p^{\prime}}}}\right\}$ are bounded in $L^{p^{\prime}}(\om)$ and $\left\{ \frac{F(u_k(x) -u_k(y))}{|x-y|^{\frac{N +sp}{p^{\prime}}}} \right\}$ is bounded in $L^{p^{\prime}}(\mathbb{R}^{2N})$. So, all the $3$ sequences will converges to some weak limits. But, as weak and pointwise limit coincides, we have 
\begin{equation}\label{ps_eq5}
    \begin{split}
        \int_{\om} |\grad u_k|^{p-2} \grad u_k \grad u &\to \int_{\om} |\grad u|^p,\\
        \int_{\om} \frac{|u_k|^{p-2}u_k u}{|x|^{p \theta}} dx &\to \int_{\om} \frac{|u|^p}
        {|x|^{p \theta}} dx,\\
        \rnnn \frac{F(u_k(x) - u_k(y))(u(x) - u(y))}{|x-y|^{N+sp}} &\to \rnnn \frac{|u(x) - u(y)|^p}{|x-y|^{N+sp}},
    \end{split}
\end{equation}
as $k \to \infty$.
Using the mixed interpolated Hardy inequality \eqref{interpolated_hardy_with_theta}, \eqref{ps_eq4}, \eqref{ps_eq5}, and \eqref{ps_eq3}, we obtain
\begin{align*}
    \left( 1 - \frac{\mu \xi}{C_H}\right)& \int_{\om} | \grad (u_k - u)|^p + \left( 1 - \frac{\mu (1 - \xi)}{C_{N,p,s}}\right)\rnnn \frac{|(u_k - u)(x) - (u_k- u)(y)|^{p}}{|x - y|^{N + p s}}\\
    &\leq  \int_{\om} | \grad (u_k - u)|^p + \rnnn \frac{|(u_k - u)(x) - (u_k- u)(y)|^{p}}{|x - y|^{N + p s}} \, dx \, dy - \mu \int_{\om} \frac{|u_k - u|^p}{|x|^{p \theta}} \\
   & \leq \int_{\om} |\grad u_k|^p - \int_{\om} |\grad u|^p + [u_k]_s^p - [u]_s^p + \int_{\om} \frac{|u_k|^p}{|x|^{p \theta}} - \int_{\om} \frac{|u|^p}{|x|^{p \theta}} + o_k(1) \\
    &= \langle \mathcal{J}^{\prime} (u_k),u_k -u \rangle + o_k(1).
\end{align*}
This implies that the sequence $\{u_k\}$ converges strongly to $u$ in $X(\om)$ as $k \to \infty$.
\end{proof}

%%%%%%%%%%%%%%%%%%%%%%%%%%%%%%%%%%%%%%%%%%%%%%%%%%%%%%%%
%%%%%%%%%   New Section
\section{Eigenvalue Problem for \texorpdfstring{$\mathcal{T}$}{T}}
In this section, we establish the existence of the first eigenvalue of the operator $\mathcal{T}$ and study some of its properties. We recall that $\mathcal{T}$ is defined as 
\begin{equation*}
    \mathcal{T}(u) = (- \plap + \pfrac) u - \mu \frac{|u|^{p-2}u}{|x|^{p \theta}}.
\end{equation*}

\noi Firstly, we would like to show that a sequence of eigenvalues of the operator $\mathcal{T}$ exists. For this, we employ the index theory. Take the class
\begin{equation*}
    \mathcal{A} = \{ A \subset X(\om) : A \mbox{ is closed}, A = - A\}.
\end{equation*}
In this class, the genus is defined for $A \in \mathcal{A}$ as 
\begin{equation*}\label{genus}
\begin{split}
\gamma^* (A) = 
\begin{cases}
0 & \mbox{ if } A = \emptyset, \\
\inf\{ m \in \ntrl \cup \{ 0 \} ; \mbox{ there exists } h \in C(A;\mathbb{R}^m \setminus \{ 0 \}), h(-u) = h(u) \}  & \mbox{ if } A \neq \emptyset,
\end{cases}
\end{split}
\end{equation*}
if no such infimum exists then we take $\gamma^*(A) = \infty $. Now, we have our theorem.
\begin{theorem}\label{existence_of_eigenvalues}
Let $2 \leq p <N$. Then, there exists a sequence of increasing eigenvalues $\la_k$ of the operator $\mathcal{T}$ with the characterization 
    \begin{equation*}
        \la_k = \inf_{K \subset \mathcal{A}_k} \sup_{u \in K} \mathcal{J}(u)
    \end{equation*}
    where $\mathcal{A}_k = \{ K \in \mathcal{A} : K \subset S, \gamma^*(K) \geq k \}$.
\end{theorem}
\begin{proof}
Clearly, $\mathcal{J}$ is even, and $S$ is a complete, symmetric, and $C^{1,1}$-manifold in $X(\om)$. Additionally, $\mathcal{J}$ is bounded below on $S$ and satisfies the (P.-S.) condition on $S$ (see Lemma \ref{ps_condition} for $d = 0$). Then, the existence of the eigenvalues $\la_k$ with the given characterization comes from Theorem \ref{thm:existence_ev_struwe}.
\end{proof}

\noi Now, we prove some results specific to the first eigenvalue.
\begin{theorem}
Let $1 < p < N$. The first eigenvalue defined in (\ref{first_eigenvalue}) is achieved by a nonnegative function $\phi_1$. In fact, $\phi_1 > 0$.
\end{theorem}
\begin{proof}
Let $\{ u_k \}$ be the minimizing sequence in $S$ such that $\mathcal{J}(u_k) \to \lambda_1$ as $k \to \infty$. Then clearly, $\mathcal{J}(u_k)$ is bounded. Also, mixed interpolated Hardy inequality \eqref{interpolated_hardy} yields $$C \geq \mathcal{J}(u_k) \geq \left( 1 - \frac{\mu \xi}{C_H}\right) \int_{\om} |\grad u_k|^p + \left( 1 - \frac{\mu(1-\xi)}{C_{N,p,s}} \right)  \int_{\rnn} \int_{\rnn} \frac{|u_k(x) - u_k(y)|^p}{|x-y|^{N+sp}}~dx ~ dy,$$ which implies $\{u_k\}$ is a bounded sequence in $X(\om)$. Now, using the Ekeland variational principle, there exists a sequence $\{ v_k \}$ such that 
\begin{equation*}
    \mathcal{J}(v_k) \leq \mathcal{J}(u_k), \; \lv u_k - v_k \rv \leq \frac{1}{k},
\end{equation*}
\begin{equation*}
    \mathcal{J}(v_k) \leq \mathcal{J}(u) + \frac{1}{k} \lv v_k - v \rv \; \mbox{ for all } u \in S.
\end{equation*}
 Next, arguing as in Theorem $3.1$ of \cite{sarika_fucik_pfractional}, we obtain $\{ v_k \}$ to be the Palais-Smale sequence. Moreover, from the choice of $\{ v_k \}$, it is bounded and up to a subsequence $v_k \weak \phi_1$ weakly in $X(\om)$. From Lemma \ref{pointwise_gradient}, we get pointwise convergence of the gradients. Also, the weak convergence of $\{v_k\}$ gives $\langle \mathcal{J}^{\prime}(\phi_1), v_k - \phi_1 \rangle \to 0$ and from definition, we have $\langle \mathcal{J}^{\prime}(v_k), v_k - \phi_1 \rangle \to 0$ as $k \to \infty$. Thus, we deduce that 
\begin{align*}
    \left| \int_{\om} |\grad v_k|^{p-2} \grad v_k \grad( v_k - \phi_1) \right.\notag  + \int_{\rnn} \int_{\rnn} \frac{F(v_k(x) - v_k(y))((v_k(x) - \phi_1(x)) - (v_k(y) - \phi_1(y))) }{|x-y|^{N+sp}}\notag \\
    \left.- \mu \int_{\om} \frac{|v_k|^{p-2}v_k (v_k - \phi_1)}{|x|^{p\theta}}\right| \leq O(\e_k) + |t_k| \lv v_k\rv_p^{p-1} \lv v_k - \phi_1 \rv_p \to 0
\end{align*}
as $k \to \infty$, where $F(t) = |t|^{p-2} t$. Thus, by the Brezis-Lieb Lemma, we have 
 \begin{equation*}
 \begin{split}
\int_{\om} |\grad (v_k - \phi_1)|^{p} dx  &= \int_{\om} |\grad v_k |^{p} dx - \int_{\om} |\grad \phi_1|^p dx + o_k(1),\\
[v_k - \phi_1]_s^p &= [v_k]_s^p - [\phi_1]_s^p+o_k(1),\\
 \int_{\om} \frac{|v_k - \phi_1|^p}{|x|^{p\theta}} &= \int_{\om} \frac{|v_k|^p}{|x|^{p\theta}} - \int_{\om} \frac{|\phi_1|^p}{|x|^{p\theta}} +o_k(1).
\end{split}
 \end{equation*}
 % \begin{align*}
 %    \int_{\rnn} \int_{\rnn} \frac{|(v_k - \phi_1)(x) - (v_k-\phi_1)(y)|^p}{|x-y|^{N+sp}} =& \int_{\rnn} \int_{\rnn} \frac{|v_k(x) - v_k(y)|^p}{|x-y|^{N+sp}} \\
 %    & + \int_{\rnn} \int_{\rnn} \frac{|\phi_1(x) - \phi_1(y)|^p}{|x-y|^{N+sp}} + o_k(1),
 % \end{align*}
 Using the above relations and mixed interpolated Hardy inequality, we argue as in Lemma \ref{ps_condition} and deduce that 
 \begin{align*}
     o_k(1)  \geq &  \left| \left( 1 - \frac{\mu \xi}{C_H}\right) \int_{\om} |\grad v_k - \grad \phi_1|^p \right. \\
     &\left. +  \left( 1 - \frac{\mu(1-\xi)}{C_{N,p,s}} \right)  \int_{\rnn} \int_{\rnn} \frac{|(v_k(x) - \phi_1(x)) - (v_k(y)- \phi_1(y))|^p}{|x-y|^{N+sp}}\, dx dy \right|.
 \end{align*}
 Thus, the sequence $\{v_k\}$ converges strongly to $\phi_1$ as $k \to \infty$. Hence, we obtain 
 \begin{equation*}
      - \plap \phi_1 + \pfrac \phi_1 - \mu \frac{|\phi_1|^{p-2}\phi_1}{|x|^{p \theta}} = \lambda_1 |\phi_1|^{p-1} \phi_1 \mbox{ in } {C_{c}^{\infty}}(\om).
 \end{equation*}
Moreover, we have $\int_{\om} |\phi_1|^p dx = 1$. This implies $\phi_1 \neq 0$. As the fractional term is present in the functional, if $\phi_1$ is not of constant sign, by the inequality
\begin{align*}
    \big||\phi_1(x)| - |\phi_1(y)|\big| \leq | \phi_1(x) - \phi_1(y)|,
\end{align*}
we have $\mathcal{J}(|\phi_1|) < \mathcal{J}(\phi_1)$. But, as $\la_1 \leq \mathcal{J}(|\phi_1|) < \mathcal{J}(\phi_1) = \la_1 $. So, $\phi_1$ has to be nonnegative. \\
Next, to prove positivity of $\phi_1$, observe that for any nonnegative $w \in \sobop(\om)$ we have 
\begin{align*} 
\int_{\om} |\grad \phi_1|^{p-2} \grad \phi_1 \grad w + \rnnn \frac{F(\phi_1(x) - \phi_1(y))(w(x) - w(y))~ dx dy}{|x-y|^{N+ps}} \\
= \int_{\om} \frac{\phi_1^{p-1} w}{|x|^{p \theta}} + \la_1 \int_{\om} \phi_1^{p-1} w ~ dx \geq 0.
\end{align*}
Now, applying Theorem $8.4$ of \cite{garain2022regularity} to obtain the following inequality
\begin{equation*}
    \left( \frac{1}{|B(x_0,\frac{r}{2})|} \int_{B(x_0,\frac{r}{2})} u^l \right)^{\frac{1}{l}} \leq C  ~ \underset{B(x_0,r)}{\mbox{ess inf}} ~ u, 
\end{equation*}
where $B(x_0,r) \subset \om,  r \in (0,1)$ and $l > 0$ is some positive number. This gives positivity of $\phi_1$. 
\end{proof}
%%%%%%%%%%%%%%%%%%%%%%%%%%%%%%%%%
\noi Next, we prove that the eigenfunctions associated with the first eigenvalue are constant multiple of each other.
\begin{theorem}
The first eigenvalue $\lambda_1$ defined in (\ref{first_eigenvalue}) is simple.
\end{theorem}
\begin{proof}
Taking $\phi_1$ and $\phi_2$ to be two positive eigenfunctions associated with the first eigenvalue $\lambda_1$. Consider $\{ \psi_k \}$ be to a sequence of functions in $C_c^{\infty}(\real)$ such that $\psi_k \geq 0, \psi_k(x) = 0 $ in $\rnn \setminus \om, \psi_k \to \phi_1$ in $X(\om)$ and convergent a.e. in $\om$. Then we have 
    \begin{align} \label{simple_eqn1}
        0 = & \int_{\om} |\grad \phi_1|^p + \int_{\rnn} \int_{\rnn} \frac{|\phi_1(x) - \phi_1(y)|^p}{|x - y|^{N + p s}} dx dy  - \mu \int_{\Omega} \frac{|\phi_1|^p}{|x|^{p \theta}} dx - \lambda_1 \int_{\Omega} |\phi_1|^p dx \notag \\
        = & \lim_{k \to \infty} \left[ \int_{\om} |\grad \psi_k|^p +  \int_{\Omega} \frac{|\psi_{k}(x) - \psi_k(y)|^p}{|x - y|^{N+ps}} dx dy 
    - \mu \int_{\Omega} \frac{|\psi_k|^p}{|x|^{p \theta}} dx - \lambda_1 \int_{\Omega}  |\psi_k|^p dx \right].
    \end{align}

\noi Let $w_k := \frac{\psi_k^p}{\left( \phi_2 + \frac{1}{k} \right)^{p - 1}}$. Then, we show that $w_k \in X(\om)$. Now, as in \cite{sarika_fucik_pfractional}, we deduce

\begin{align*}
    &|w_k(x) - w_k(y)| \\
    &\leq \left| \frac{\psi_k^p(x)}{\left( \phi_2 + \frac{1}{k} \right)^{p - 1}(x)} - \frac{\psi_k^p(x)}{\left( \phi_2 + \frac{1}{k} \right)^{p - 1}(y)} \right| 
    + \left| \frac{\psi_k^p(x)}{\left( \phi_2 + \frac{1}{k} \right)^{p - 1}(y)} - \frac{\psi_k^p(y)}{\left( \phi_2 + \frac{1}{k} \right)^{p - 1}(y)} \right| \\
    &\leq 2 p k^{p-1} \|\psi_k\|_\infty^{p - 1} |\psi_k(x) - \psi_k(y)| + \|\psi_k\|_\infty^p (p - 1) \left( \frac{|\phi_2(x) - \phi_2(y)|}{\left( \phi_2 + \frac{1}{k} \right)^{p - 1}(x)} + \frac{|\phi_2(x) - \phi_2(y)|}{\left( \phi_2 + \frac{1}{k} \right)^{p - 1}(y)} \right) \\
    &\leq C_{k,p,\lv \psi_k \rv_{\infty}} ( | \psi_k(x) - \psi_k(y)| + |\phi_2(x) - \phi_2(y)|)
\end{align*}

\noi for all $(x, y) \in \rnn \times \rnn$. Also, we have 
\begin{align*}
    \grad w_k = \frac{p \psi_k^{p-1} \grad \psi_k}{\left( \phi_2 + \frac{1}{k} \right)^{p-1}} - \frac{\psi_k^p (p-1) \grad \phi_2}{\left( \phi_2 + \frac{1}{k} \right)^p}.
\end{align*}
This implies 
\begin{align*}
    |\grad w_k|^p \leq 2^{p-1} \left( p^p k^{p(p-1)} \lv \psi_k \rv_{\infty}^{p(p-1)} |\grad \psi_k|^p + (p-1)^p k^{p} \lv \psi_k\rv_{\infty}^{p^2} | \grad \phi_2 |^p \right).
\end{align*}

\noi Hence, $w_k \in X(\om)$ for all $k \in \mathbb{N}$, as $\psi_k, \phi_2 \in X(\om)$. Now, testing equation (\ref{spectrum_equation}) satisfied by $\phi_2$ with $w_k$, we obtain
\begin{align*}
    \int_{\om} |\grad \phi_2|^{p-2} \grad \phi_2 \grad w_k dx + \int_{\rnn} \int_{\rnn}  \frac{F(\phi_2(x) - \phi_2(y)) \left( \frac{\psi_k^p(x)}{\left( \phi_2 + \frac{1}{k} \right)^{p - 1}(x)} - \frac{\psi_k^p(y)}{\left( \phi_2 + \frac{1}{k} \right)^{p - 1}(y)} \right)}{|x - y|^{N + p s}} dy ~ dx &\\
    = \int_{\Omega} \left( \lambda_1 + \frac{\mu}{|x|^{p \theta}} \right) \frac{\phi_2^{p-1} \psi_k^p}{\left( \phi_2 + \frac{1}{k} \right)^{p - 1}} dx.& 
\end{align*}
Using above and equation (\ref{simple_eqn1}), we deduce 
\begin{align*}
    0 = & \lim_{k \to \infty} \left( \int_{\om} L_c(\psi_k, \phi_2) + \int_{\rnn} \int_{\rnn} L(\psi_k, \phi_2)\right) ~dx dy \\
    \geq & \left( \int_{\om} L_c(\phi_1, \phi_2) + \int_{\rnn} \int_{\rnn} L(\phi_1, \phi_2)\right) ~dx dy \geq 0.
\end{align*}
So, the continuous and discrete version of Picone identity in Proposition \ref{continuous_picone_identity} and \ref{discrete_picone_identity} implies that $\phi_1 = c \phi_2$ a.e. for some constant $c$.  Therefore, $\lambda_1$ is simple.
\end{proof}
\noi Now we show that in contrast to the first eigenfunction being of constant sign, eigenfunctions corresponding to other eigenvalues change sign. 
\begin{theorem}\label{sign_changing_eigenfunctions}
    If $w$ is an eigenfunction corresponding to eigenvalue $\lambda$ such that $\lambda \neq \lambda_1$. Then $w$ changes sign. 
\end{theorem}
\begin{proof}
Let $\phi_1$ be the eigenfunction corresponding to the first eigenvalue $\lambda_1$. Then $\phi_1$ and $w$ satisfy
    \begin{equation}\label{sign_eqn1}
        -\plap \phi_1 + \pfrac \phi_1 - \mu \frac{|\phi_1|^{p-2} \phi_1}{|x|^{p \theta}} = \lambda_1  \phi_1^{p-1} \quad \text{in} \, \mathcal{D}'(\Omega) 
    \end{equation}
\begin{equation}\label{sign_eqn2}
    -\plap w + \pfrac w - \mu \frac{|w|^{p-2} w}{|x|^{p \theta}} = \lambda  |w|^{p-2} w \quad \text{in} \, \mathcal{D}'(\Omega)
\end{equation}
respectively. Suppose $w$ does not change the sign. Then, we may assume $w \geq 0$. Let $\{\psi_k\}_{k \geq 2}$ be a sequence in $C^\infty_c (\rnn)$ such that $\psi_k = 0$ in $\mathbb{R}^n \setminus \Omega$, $\psi_k \to \phi_1$ in $X(\om)$ as $k \to \infty$. Now we consider the test functions $w_1 = \phi_1$, $w_k = \frac{\psi_k^p}{\left( w + \frac{1}{k} \right)^{p-1}}$. Then $w_1, w_k \in X(\om)$. Taking $w_1$ and $w_k$ as test functions in \eqref{sign_eqn1} and \eqref{sign_eqn2} respectively, we obtain
\begin{equation}\label{sign_eqn3}
     \int_{\om} |\grad \phi_1|^p +  \int_{\rnn} \int_{\rnn} \frac{|\phi_1(x) - \phi_1(y)|^p}{|x - y|^{N + ps}} dx dy 
    - \mu \int_\Omega \frac{|\phi_1|^p}{|x|^{p \theta}} dx 
    = \lambda_1 \int_\Omega  |\phi_1|^p dx 
    \end{equation}
   \small{ \begin{align*}
    \int_{\om} |\grad w|^{p-2} \grad w \grad \left( \frac{\psi_k^p}{\left( w + \frac{1}{k} \right)^{p-1}} \right) + \int_{\rnn} \int_{\rnn} \frac{F(w(x) - w(y))}{|x - y|^{N + ps}} \left( \frac{\psi_k^p(x)}{\left( w + \frac{1}{k} \right)^{p-1}(x)} - \frac{\psi_k^p(y)}{\left( w + \frac{1}{k} \right)^{p-1}(y)} \right)  \\
    - \mu \int_\Omega \frac{|w|^{p-2} w}{|x|^{p \theta}} \frac{\psi_k^p}{\left( w + \frac{1}{k} \right)^{p-1}} dx 
    = \lambda \int_\Omega   |w|^{p-2} w \frac{\psi_k^p}{\left( w + \frac{1}{k} \right)^{p-1}} dx.
\end{align*}}

\noi Since $L(\psi_k, w + \frac{1}{k}) \geq 0$ and $L_c(\psi_k, w + \frac{1}{k}) \geq 0$, we have
\begin{equation}\label{sign_eqn4}
\int_{\om} |\grad \psi_k|^p + \int_{\rnn} \int_{\rnn} \frac{|\psi_k(x) - \psi_k(y)|^p}{|x - y|^{N + ps}} dx dy 
- \int_\Omega \left( \lambda   + \frac{\mu}{|x|^{p \theta}} \right) \psi_k^p \left( \frac{w}{w + \frac{1}{k}} \right)^{p-1} dx \geq 0. 
\end{equation}

\noi Subtracting \eqref{sign_eqn3} from \eqref{sign_eqn4} and taking the limit as $k \to \infty$, we obtain

\[
(\lambda - \lambda_1) \int_\Omega   |\phi_1|^p \leq 0,
\]
which gives a contradiction to the fact that $\lambda > \lambda_1$. Hence, $w$ changes sign. 
\end{proof}
\begin{theorem}
    Let $1 < p < N$. Then $\lambda_1$ is isolated in the spectrum of $\mathcal{T}$.
\end{theorem}
\begin{proof}
On contrary, assume that $\lambda_1$ is not isolated in the spectrum. Then there exists a sequence of eigenfunctions $\{\psi_k\}_{k\geq 1}$ in $X(\om)$ corresponding to a sequence of eigenvalues $\{ \mu_k \}_{k\geq 1}$ satisfying:
\begin{equation}\label{isolated_eq1_ev}
    - \plap \psi_k +  \pfrac \psi_k - \mu \frac{|\psi_k|^{p-2} \psi_k}{|x|^{p\theta}} = \mu_k  |\psi_k|^{p-2} \psi_k \quad \text{in} \ \Omega, \quad \psi_k = 0 \ \text{in} \ \mathbb{R}^n \setminus \Omega,
\end{equation}
with $\|\psi_k\|_p = 1$ and $\mu_k \to \la_1$ as $k \to \infty$. Moreover, we have
\[
\int_{\om}|\grad \psi_k|^p + \int_{\rnn} \int_{\rnn} \frac{|\psi_k(x) - \psi_k(y)|^p}{|x - y|^{N+ps}} \, dx \, dy - \mu \int_\Omega \frac{|\psi_k|^p}{|x|^{p\theta}} \, dx = \mu_k \int_\Omega  |\psi_k|^p \, dx = \mu_k.
\]
The mixed interpolated Hardy inequality \eqref{interpolated_hardy} yields that $\{ \psi_k \}$ is a bounded sequence in $X(\om)$. Thus, up to a subsequence, we have $\psi_k \weak \phi$ weakly in $X(\om)$ and $\psi_k \to \phi$ strongly in $L^p(\om)$ as $k \to \infty$. Then, Lemma \ref{pointwise_gradient} gives pointwise convergence of the gradients. Hence, 
\[
- \plap \phi + \pfrac \phi - \mu \frac{|\phi|^{p-2} \phi}{|x|^{p\theta}} = \lambda_1  |\phi|^{p-2} \phi \quad \text{in} \ \Omega, \quad \phi = 0 \ \text{in} \ \mathbb{R}^n \setminus \Omega.
\]
Therefore, \( \phi = \pm \phi_1 \). Since \( \psi_k \) changes sign by Theorem \ref{sign_changing_eigenfunctions}, without loss of generality, we can assume that \( u = +\phi_1 \), then
\begin{equation}\label{isolated_goes_to_0}
|\{x : \psi_k(x) < 0\}| \to 0 \mbox{ as } k \to \infty.
\end{equation}
Putting \( \psi_k^- \) as a test function in \eqref{isolated_eq1_ev}, we get
\begin{align*}
-\int_{\om} | \grad \psi_k^- |^{p}  + \rnnn \frac{F(\psi_k(x) - \psi_k(y))(\psi_k^-(x) - \psi_k^-(y))}{|x - y|^{N+ps}}  
+ \mu \int_\Omega \frac{(\psi_k^-)^{p}}{|x|^{p\theta}} \, 
= \mu_k \int_\Omega  (\psi_k^-)^p \, .
\end{align*}
% Now using the facts 
% \begin{equation}\label{isolated_basic_relation_1}
%     ((\psi_k(x) - \psi_k(y))(\psi_k^-(x) - \psi_k^-(y)) = -\psi_k^-(x)\psi_k^+(y) - \psi_k^-(y)\psi_k^+(x) - (\psi_k^-(x) - \psi_k^-(y))^2
% \end{equation}
% and for $p \geq 2$, 
% \begin{equation}\label{isolated_ev_basic_relation}
% \begin{split}
%     |\psi_k(x) &- \psi_k(y)|^{p-2} \\
%     &= (|\psi_k(x) - \psi_k(y)|^2)^{\frac{p-2}{2}}  \\
%     & =  (|\psi_k^+(x) - \psi_k^+(y)|^2 + |\psi_k^-(x) - \psi_k^-(y)|^2 + 2 \psi_k^+(x) \psi_k^-(y) + 2 \psi_k^+(y) \psi_k^-(x))^{\frac{p-2}{2}}  \\
%     & \geq  |\psi_k^-(x) - \psi_k^-(y)|^{p-2},
% \end{split}
% \end{equation}
Using \eqref{brasco_inequality_negative}, we obtain
\begin{align*}
\int_{\om} | \grad \psi_k^- |^{p} +& \rnnn  \frac{|\psi_k^-(x) - \psi_k^-(y)|^p}{|x - y|^{N+p s}} dx \, dy - \mu \int_\Omega \frac{|\psi_k^-|^p}{|x|^{p\theta}} dx \\ 
\leq &\int_{\om} | \grad \psi_k^- |^{p} - \rnnn \frac{F(\psi_k(x) - \psi_k(y))(\psi_k^-(x) - \psi_k^-(y))}{|x - y|^{N+p s}} dx \, dy - \mu \int_\Omega \frac{|\psi_k^-|^p}{|x|^{p\theta}} dx \\ 
= &\mu_k \int_\Omega |\psi_k^-|^p dx.
\end{align*}
Thus, mixed interpolated Hardy inequality \eqref{interpolated_hardy}, H\"{o}lder inequality, and Sobolev inequality yield
\begin{align*}
    \left( 1 - \frac{\mu \xi}{C_H}\right) \int_{\om}| \grad \psi_k^- |^{p} &
    \leq \left( 1 - \frac{\mu \xi}{C_H}\right) \int_{\om}| \grad \psi_k^- |^{p} + \left( 1 - \frac{\mu (1 - \xi)}{C_{N,p,s}}\right) \rnnn \frac{|\psi_k^-(x) - \psi_k^-(y)|^p}{|x - y|^{N+p s}}    \\
    &\leq \int_{\om} | \grad \psi_k^- |^{p} + \rnnn \frac{|\psi_k^-(x) - \psi_k^-(y)|^p}{|x - y|^{N+p s}} dx \, dy - \mu \int_\Omega \frac{|\psi_k^-|^p}{|x|^{p\theta}} dx \\
    & \leq \mu_k \left( \int_{\om} | \psi_k^-|^{p^*} \right)^{p/p^*} | \{ x : \psi_k(x) < 0 \}|^{1 - p/p^*}\\
    &\leq \mu_k C  \left(\int_{\om}| \grad \psi_k^- |^{p} \right) | \{ x : \psi_k(x) < 0 \}|^{1 - p/p^*},
\end{align*}
where $p^* = pN/(N -p)$. It implies that 
$$|\{ x : \psi_k(x) < 0 \}|^{1 - p/p^*} \geq \left( 1 - \frac{\mu \xi}{C_H}\right) \mu_k^{-1} C^{-1}, $$
which leads to a contradiction with \eqref{isolated_goes_to_0}. Hence, \( \lambda_1 \) is isolated in the spectrum of \( \mathcal{T} \).
\end{proof}

\section{\fucik Spectrum}
In this section, we study the \fucik spectrum, denoted by $\displaystyle \sum_{p,\mu,\theta}$, of \eqref{pq}. It is defined as 
\begin{equation*}
    \sum_{p,\mu,\theta} := \{ (\alpha, \beta) \in \mathbb{R}^2: \eqref{pq} \mbox{ has a nontrivial solution}   \}. 
\end{equation*}
Next, we show that the points in $\sum_{p,\mu,\theta}$ are associated with the critical value of some restricted functional. For nonnegative real $d$, recalling the functional $\mathcal{J}_d: X(\om) \to \mathbb{R}$ given by
\begin{align*}
\mathcal{J}_d(u) =  \int_{\om} |\grad u|^p \, dx + \rnnn \frac{|u(x)-u(y)|^{p}}{|x-y|^{N+ps}} \, dx \, dy - \mu \int_{\Omega} \frac{|u|^p}{|x|^{p\theta}} \, dx - d \int_{\Omega} (u^+)^p \, dx.
\end{align*}
Also, $\mathcal{J}_d \in C^1(X(\om), \mathbb{R})$ and for any $ v \in X(\om)$
\begin{align*}
\frac{1}{p}\langle \mathcal{J}_d'(u), v \rangle =& \int_{\om} |\grad u|^{p-2} \grad u \grad v + \rnnn \frac{|u(x)-u(y)|^{p-2}(u(x)-u(y))(v(x)-v(y))}{|x-y|^{N+p s}} \, dx \, dy \\
 & - \mu \int_{\Omega} \frac{|u|^{p-2} u v}{|x|^{p\theta}} \, dx - d \int_{\Omega} (u^+)^{p-1} v \, dx.
\end{align*}
Moreover, $\tilde{\mathcal{J}}_d := \mathcal{J}_d|_S \in C^1(X(\om), \mathbb{R})$, where $S$ is defined as
\[
S := \left\{ u \in X(\om) : I(u) := \int_{\Omega} |u|^p = 1 \right\}.
\]
\begin{lemma}
For $d \geq 0$, $(d + t, t) \in \mathbb{R}^2$ belongs to the spectrum $\sum_{p,\mu,\theta}$ if and only if there exists a critical point $u \in S$ of $\tilde{\mathcal{J}}_d$ such that $t = \tilde{\mathcal{J}}_d(u)$, a critical value.
\end{lemma}
\begin{proof}
From Lagrange theorem, we deduce that $u \in S$ is a critical point of $\tilde{\mathcal{J}}_d$ if and only if there exists $t \in \mathbb{R}$ such that
\begin{equation}\label{lagrange_eq}
\langle \mathcal{J}_d'(u), v \rangle  = tp \int_{\Omega} |u|^{p-2} u v,
\end{equation}
for all $v \in X(\om)$. Hence, $u \in S$ is a nontrivial weak solution to the problem
\[
-\plap u + \pfrac u - \mu \frac{|u|^{p-2} u}{|x|^{p \theta}} = (d + t)(u^+)^{p-1} + t(u^-)^{p-1} \text{ in } D'(\Omega), \quad u = 0 \text{ on } \mathbb{R}^n \setminus \Omega,
\]
which exactly means $(d + t, t) \in \displaystyle\sum_{p,\mu,\theta}$. Putting $v = u$ in \eqref{lagrange_eq}, we deduce $t = \tilde{\mathcal{J}}_d(u)$.
\end{proof}

\noi Next, we show that the vertical and horizontal lines that passes through $(\la_1,\la_1)$ lies in the \fucik spectrum. For this, we have the following propositions.
\begin{proposition}
For $d \geq 0$, the first eigenfunction $\phi_1$ is a global minimum for $\mathcal{J}_d$ with $\mathcal{J}_d(\phi_1) = \lambda_1 - d$. The corresponding point in $\Sigma_{p,\mu,\theta}$ is $(\lambda_1, \lambda_1 - d)$, which lies on the vertical line through $(\lambda_1, \lambda_1)$.
\end{proposition}
\begin{proof}
Consider
\begin{align*}
\mathcal{J}_d(u) &= \int_{\om} |\grad u|^p + \rnnn \frac{|u(x) - u(y)|^p}{|x - y|^{N+ps}} \, dx \, dy - \mu \int_{\Omega} \frac{|u|^p}{|x|^{p \theta}} \, dx - d \int_{\Omega} (u^+)^p \, dx \\
&\geq \lambda_1 \int_{\Omega} |u|^p \, dx - d \int_{\Omega} (u^+)^p \, dx \geq \lambda_1 - d.
\end{align*}
Then, $\mathcal{J}_d$ is bounded below by $\lambda_1 - d$. Also,
\begin{align*}
\mathcal{J}_d(\phi_1) &= \lambda_1 - d\int_{\Omega} \phi_1^p \, dx = \lambda_1 - d.
\end{align*}
Therefore, $\phi_1$ is a global minimum of $\mathcal{J}_d$ with $\mathcal{J}_d(\phi_1) = \lambda_1 - d$.
\end{proof}

\noi Another set of points in the \fucik spectrum comes from a second critical point of $\mathcal{J}_d$ at $-\phi_1$.

\begin{proposition}
The negative eigenfunction $-\phi_1$ is a strict local minimum for $\mathcal{J}_d$ with $\mathcal{J}_d(-\phi_1) = \lambda_1$. The corresponding point in $\Sigma_{p,\mu, \theta}$ is $(\lambda_1 + d, \lambda_1)$.
\end{proposition}
\begin{proof}
On the contrary, if $-\phi_1$ is not the strict local minima, then there exists a sequence $u_k \in S$, $u_k \neq -\phi_1$ with $\mathcal{J}_d(u_k) \leq \lambda_1$, $u_k \to -\phi_1$. Since $u_k \neq -\phi_1$, then $| \{x \in \om : ~ u_k(x)  < 0 \}| > 0$ for sufficiently large $k$. If $u_k \leq 0$ for a.e $x \in \Omega$, then

\[
\mathcal{J}_d(u_k) =  \int_{\om} |\grad u_k|^p + \rnnn \frac{|u_k(x) - u_k(y)|^p}{|x - y|^{N+ps}} \, dx \, dy - \mu \int_{\Omega} \frac{|u_k|^p}{|x|^{p \theta}} \, dx > \lambda_1,
\]

\noi as $u_k \neq \pm \phi_1$. This gives a contradiction as $\mathcal{J}_d(u_k) \leq \lambda_1$. So $u_k$ changes sign for sufficiently large $k$.\\
Note that $(a+b)^p \geq a^p + b^p$ for $a,b \geq 0$ and $p > 1$, we have
\begin{align*}
\rnnn & \frac{|u_k(x) - u_k(y)|^p}{|x - y|^{N+p s}} \\
= &  \iint_{{\begin{array}{l} \scriptstyle{\{ u_k(x) > 0,}\\\scriptstyle{ u_k(y) > 0\}}\end{array}}}  \frac{|u_k^+(x) - u_k^+(y)|^p}{|x - y|^{N+p s}}  + \iint_{{\begin{array}{l}\scriptscriptstyle{\{ u_k(x) > 0,}\\\scriptscriptstyle{ u_k(y) < 0\}}\end{array}}}  \frac{|u_k^+(x) + u_k^-(y)|^p}{|x - y|^{N+p s}} \\
& + \iint_{{\begin{array}{l}\scriptscriptstyle{\{ u_k(x) < 0,}\\ \scriptscriptstyle{u_k(y) > 0\}}\end{array}}}  \frac{|u_k^-(x) + u_k^+(y)|^p}{|x - y|^{N+p s}} + \iint_{{\begin{array}{l}\scriptscriptstyle{\{ u_k(x) < 0,}\\\scriptscriptstyle{ u_k(y) < 0\}}\end{array}}}  \frac{|u_k^-(x) - u_k^-(y)|^p}{|x - y|^{N+p s}} \\
\geq & \iint_{{\begin{array}{l}\scriptscriptstyle{\{ u_k(x) > 0,}\\ \scriptscriptstyle{u_k(y) > 0\}}\end{array}}}  \frac{|u_k^+(x) - u_k^+(y)|^p}{|x - y|^{N+p s}}  + \iint_{{\begin{array}{l} \scriptscriptstyle{\{ u_k(x) > 0,}\\\scriptscriptstyle{ u_k(y) < 0\}}\end{array}}}  \frac{|u_k^+(x)|^p + |u_k^-(y)|^p}{|x - y|^{N+p s}} \\
& + \iint_{{\begin{array}{l}\scriptscriptstyle{\{ u_k(x) < 0,}\\\scriptscriptstyle{ u_k(y) > 0\}}\end{array}}}  \frac{|u_k^-(x)|^p + |u_k^+(y)|^p}{|x - y|^{N+p s}} + \iint_{{\begin{array}{l}\scriptscriptstyle{\{ u_k(x) < 0,}\\\scriptscriptstyle{ u_k(y) < 0\}}\end{array}}}  \frac{|u_k^-(x) - u_k^-(y)|^p}{|x - y|^{N+p s}} \\
= & \rnnn \frac{|u_k^+(x) - u_k^+(y)|^p}{|x - y|^{N+p s}} + \rnnn \frac{|u_k^-(x) - u_k^-(y)|^p}{|x - y|^{N+p s}}.
\end{align*}
Now, using the above, mixed interpolated Hardy inequality \eqref{interpolated_hardy} and definition of $\la_1$, we have
\begin{align*}
\mathcal{J}_d(u_k) \geq & \int_{\om} |\grad u_k^+ |^p + \int_{\om} |\grad u_k^- |^p  +  \rnnn \frac{|u_k^+(x) - u_k^+(y)|^p}{|x - y|^{N+p s}} \\
&  + \rnnn \frac{|u_k^-(x) - u_k^-(y)|^p}{|x - y|^{N+p s}}- \mu \int_{\Omega} \frac{|u_k^+|^p}{|x|^{p \theta}} \, dx - \mu \int_{\Omega} \frac{|u_k^-|^p}{|x|^{p \theta}} \, dx - d \int_{\Omega} (u_k^+)^p \, dx\\
\geq & \left( 1 - \frac{\mu \xi}{C_H}\right)\int_{\om} |\grad u_k^+ |^p + \la_1 \int_{\om} (u_k^-)^p -  d\int_{\Omega} (u_k^+)^p \, dx \\
= & \left[ r_k \left( 1 - \frac{\mu \xi}{C_H}\right) - d  \right]\int_{\Omega} (u_k^+)^p \, dx + \la_1 \int_{\om} (u_k^-)^p,
\end{align*}
with $r_k = \int_{\om} |\grad u_k^+ |^p/\int_{\Omega} (u_k^+)^p $.
As $u_k \in S$, we obtain
\[
\mathcal{J}_d(u_k) \leq \lambda_1 = \lambda_1 \int_{\Omega} (u_k^+)^p \, dx + \lambda_1 \int_{\Omega} (u_k^-)^p \, dx.
\]

\noi Combining both the inequalities, we have
$$r_k \left( 1 - \frac{\mu \xi}{C_H}\right) - d \leq \la_1 $$
which contradicts to the fact that $r_k \to +\infty$ from Lemma $2.4$ of \cite{Cuesta1999fucik}, as required.
\end{proof}
\noi Next, we obtain a third critical point of $\mathcal{J}_d$ in the following lemma with the help of Proposition \ref{mountain_pass_variant} and first two critical points $\phi_1$ and $- \phi_1$ of $\mathcal{J}_d$.

\begin{lemma}
Let $\epsilon_0 > 0$ be such that
\[
\tilde{\mathcal{J}}_d(u) > \tilde{\mathcal{J}}_d(-\phi_1) 
\]
for all $u \in B(-\phi_1, \epsilon_0) \cap S$ with $u \neq -\phi_1$, where the ball is taken in $X(\om)$. Then for any $0 < \epsilon < \epsilon_0$,
\[
\inf \{\tilde{\mathcal{J}}_d(u): u \in S \text{ and } \|u - (-\phi_1)\| = \epsilon\} > \tilde{\mathcal{J}}_d(-\phi_1).
\]
\end{lemma}
\begin{proof}
The proof of the following lemma follows from the Lemma $2.9$ of \cite{Cuesta1999fucik}.
\end{proof}
% \begin{proposition}
% Let $X(\om)$ be a Banach Space. Let $\epsilon > 0$ such that $\|\phi_1 - (-\phi_1)\| > \epsilon$ and
% \[
% \inf \{\tilde{\mathcal{J}}_d(u): u \in S \text{ and } \|u - (-\phi_1)\| = \epsilon\} > \max \{\tilde{\mathcal{J}}_d(-\phi_1), \tilde{\mathcal{J}}_d(\phi_1)\}.
% \]
% Then $\Gamma = \{\gamma \in C([-1, 1], S): \gamma(-1) = -\phi_1 \text{ and } \gamma(1) = \phi_1\}$ is nonempty and

% \[
% c(d) = \inf_{\gamma \in \Gamma} \max_{u \in \gamma([-1, 1])} \mathcal{J}_d(u) \qquad (4.11)
% \]

% is a critical value of $\tilde{\mathcal{J}}_d$. Moreover $c(d) > \lambda_1$.
% \end{proposition}
% \begin{proof}
% Proof: We prove that $\Gamma$ is nonempty. To this end, we take $\phi \in X(\om)$ such that $\phi \notin \mathbb{R} \phi_1$ and consider the path $t \phi_1 + (1 - t) \phi$ and normalized it in $L^p(V, \Omega)$ norm. Moreover the (P.S) condition and the geometric assumption are satisfied by the Lemmas 4.4 and 4.5. Then by Proposition 2.4, $c(d)$ is a critical value of $\mathcal{J}_d$. Using the definition of $c(d)$ we have $c(d) > \max \{\tilde{\mathcal{J}}_d(-\phi_1), \tilde{\mathcal{J}}_d(\phi_1)\} = \lambda_1$.
% \end{proof}

\noi Now applying the Mountain pass Lemma \ref{mountain_pass_variant} with the minimax level defined as 
\begin{equation}\label{minimax_level}
    c(d) := \inf_{\gamma \in \Gamma} \max_{u \in \gamma([-1, 1])} \mathcal{J}_d(u),
\end{equation}
with $\Gamma = \{\gamma \in C([-1, 1], S): \gamma(-1) = -\phi_1 \text{ and } \gamma(1) = \phi_1\}$, we have proved the following:
\begin{theorem}\label{curve}
For each $d \geq 0$, the point $(d + c(d), c(d))$, where $c(d) > \lambda_1$ is defined by the minimax formula \eqref{minimax_level}, then the point $(d + c(d), c(d))$ belongs to $\Sigma_{p,\mu, \theta}$.
\end{theorem}
\noi It is a trivial fact that $\Sigma_{p,\mu, \theta}$ is symmetric with respect to the diagonal. The whole curve that we obtain using Theorem \ref{curve} and symmetrizing is denoted by

\[
\mathcal{C} = \{(d + c(d), c(d)), (c(d), d + c(d)): d \geq 0\}.
\]
Next we show that the trivial lines $\mathbb{R} \times \{\la_1\}$ and $\{\la_1\} \times \mathbb{R}$ are isolated in $\Sigma_{p,\mu,\theta}$ and then we state some topological properties of the functional $\mathcal{J}_d.$
% , and finally, we prove that the curve $C$ constructed in the previous section is the first nontrivial curve in the spectrum $\Sigma_{p,\mu,\theta}$.
% Moreover, we give the variational characterization of $\lambda_2$.
\begin{proposition}
Let $1 < p < N$. Then there does not exist any sequence $(a_k, b_k) \in \sum_{p,\mu,\theta}$ with $a_k > \lambda_1$ and $b_k > \lambda_1$ such that $(a_k, b_k) \to (a, b)$ with $a = \lambda_1$ or $b = \lambda_1$. Therefore, the lines $\mathbb{R} \times \{\lambda_1\}$ and $\{\lambda_1\} \times \mathbb{R}$ are isolated in $\sum_{p,\mu,\theta}$.
\end{proposition}
\begin{proof}
On the contrary, assume that there exists a sequence $(a_k, b_k) \in \sum_{p,\mu,\theta}$ with $a_k > \lambda_1, b_k > \la_1$ and $(a_k, b_k) \to (a, b)$ with $a = \lambda_1$ or $b = \lambda_1$. Without loss of generality, we can assume $a = \la_1$. Let $u_k \in X(\om)$ be a solution of
\begin{equation}\label{isolated_eq1}
-\plap u_k + \pfrac u_k - \mu \frac{|u_k|^{p-2}u_k}{|x|^{p\theta}} = a_k (u_k^+)^{p-1} - b_k (u_k^-)^{p-1} \quad \text{in } \Omega, \quad u_k = 0 \text{ on } \mathbb{R}^n \setminus \Omega
\end{equation}
with $\|u_k\|_{p} = 1$. Then, testing \eqref{isolated_eq1} with \(u_k\), we have
\begin{align*}
\int_{\om} |\grad u_k |^p + \rnnn \frac{|u_k(x) - u_k(y)|^p}{|x - y|^{N + ps}} \, dxdy - \mu \int_\Omega \frac{|u_k|^p}{|x|^{p\theta}} \, dx &= a_k \int_\Omega  (u_k^+)^p \, dx - b_k \int_\Omega  (u_k^-)^p \, dx \\
&\leq (a_k + b_k) \int_{\om} |u_k|^p = (a_k + b_k).
\end{align*}
From mixed interpolated Hardy inequality \eqref{interpolated_hardy} we deduce that $\{ u_k \}$ is a bounded sequence in $X(\om)$. Therefore up to a subsequence $u_k \rightharpoonup u$ weakly in $X(\om)$ and $u_k \to u$ strongly in $L^p(\Omega)$. Additionally, Lemma \ref{pointwise_gradient} gives that $\grad u_k(x) \to \grad u(x)$ a.e. in $\om$ as $k \to \infty$. Thus, taking limit $k \to \infty$ in \eqref{isolated_eq1}, we get
\begin{equation}\label{isolated_eq2}
-\plap u + \pfrac u - \mu \frac{|u|^{p-2} u}{|x|^{p\theta}} = \lambda_1  (u^+)^{p-1} - b (u^-)^{p-1} \quad \text{in } \Omega, \quad u = 0 \text{ on } \mathbb{R}^n \setminus \Omega.
\end{equation}
Taking $u^+$ as a test function in \eqref{isolated_eq2}, we obtain
\begin{align*}
\int_{\om} |\grad u^+|^p + \rnnn \frac{F(u(x) - u(y))(u^+(x) - u^+(y))}{|x - y|^{N + ps}} \, dx dy - \mu \int_\Omega \frac{(u^+)^p}{|x|^{p\theta}} \, dx
= \lambda_1 \int_\Omega  (u^+)^p .
\end{align*}
% Using this together with the definition of first eigenvalue, inequality \eqref{isolated_ev_basic_relation} in the form 
% \begin{equation}\label{fractional_basic_inequality}
%     |u(x) - u(y)|^{p-2} \geq |u^+(x) - u^+(y)|^{p-2},
% \end{equation} 
% \noi and the relation
% \begin{equation}\label{isolated_basic_relation_2}
%     (  u(x) -  u(y))( u^+(x) -  u^+(y)) = ( u^+(x) -  u^+(y))^2 + 2  u^+(x)  u^-(y),
% \end{equation}
% we obtain
Using the definition of $\la_1$ together with the inequality \eqref{brasco_inequality_positive} in the above, we deduce
\begin{align*}
\lambda_1 \int_\Omega  (u^+)^p   \leq  \int_{\om} | \grad u^+|^p + \rnnn \frac{|u^+(x) - u^+(y)|^p}{|x - y|^{N + ps}} \, dxdy - \mu \int_\Omega \frac{(u^+)^p}{|x|^{p\theta}} \, dx \leq  \lambda_1 \int_\Omega  (u^+)^p.
\end{align*}
Thus
\[
\int_{\om} | \grad u^+|^p + \rnnn \frac{|u^+(x) - u^+(y)|^p}{|x - y|^{N + ps}} \, dxdy - \mu \int_\Omega \frac{(u^+)^p}{|x|^{p\theta}} \, dx = \lambda_1 \int_\Omega  (u^+)^p \, dx.
\]
Hence, from the simplicity of the first eigenvalue, either $u^+ \equiv 0$ or $u = \phi_1$. If $u^+ \equiv 0$, then $u \leq 0$ and \eqref{isolated_eq1} implies that $u$ is an eigenfunction with $u \leq 0$ so that $u = -\phi_1$. So, in any case, $u_k$ converges to either $\phi_1$ or $-\phi_1$ in $L^p(\Omega)$. Thus
\begin{equation} \label{isolated_eq3}
     \text{either } |\{x \in \Omega : u_k(x) < 0\}| \to 0 \quad \text{or} \quad |\{x \in \Omega : u_k(x) > 0\}| \to 0,
\end{equation}
as $k \to \infty$. Now, putting $u_k^+$ as test function in \eqref{isolated_eq1}, we get
\begin{equation}\label{isolated_eqn_new}
\int_{\om} | \grad u_k^+|^p + \rnnn \frac{F(u_k(x) - u_k(y))(u_k^+(x) - u_k^+(y))}{|x - y|^{N + ps}}  - \mu \int_\Omega \frac{(u_k^+)^p}{|x|^{p\theta}} = a_k \int_\Omega  (u_k^+)^p .
\end{equation}
Now, using the mixed interpolated Hardy inequality, relations \eqref{isolated_eqn_new}, \eqref{brasco_inequality_positive}, H\"older's inequality, and Sobolev inequality, we get
\begin{align*}
 \left( 1 - \frac{\mu \xi}{C_H}\right) \int_{\om} |\grad u_k^+|^p \leq &  \left( 1 - \frac{\mu \xi}{C_H}\right) \int_{\om} |\grad u_k^+|^p + \left( 1 - \frac{\mu (1 - \xi)}{C_{N,p,s}} \right)\rnnn \frac{|u_k^+(x) - u_k^+(y)|^p}{|x - y|^{N + p s}}  \\
 \leq & \int_{\om} | \grad u_k^+|^p + \rnnn \frac{|u_k^+(x) - u_k^+(y)|^p}{|x - y|^{N + p s}} \, dxdy - \mu \int_\Omega \frac{(u_k^+)^p}{|x|^{p\theta}} \, dx \\
 \leq & ~a_k \int_\Omega  (u_k^+)^p \, dx \\
 \leq &~ a_k C |\{x \in \Omega : u_k(x) > 0\}|^{1 - \frac{p}{q}}\int_{\om} |\grad u_k^+|^p,
\end{align*}
with a constant $C > 0$, and $p < q \leq p^* = \frac{Np}{N - p}$. Then we have
\[
    |\{x \in \Omega : u_k(x) > 0\}|^{1 - \frac{p}{q}} \geq a_k^{-1} C^{-1} \left( 1 - \frac{\mu \xi}{C_H}\right).
\]
Similarly, one can show that
\[
    |\{x \in \Omega : u_k(x) < 0\}|^{1 - \frac{p}{q}} \geq b_k^{-1} C^{-1}\left( 1 - \frac{\mu \xi}{C_H}\right).
\]
This gives contradiction to \eqref{isolated_eq3}. Hence the trivial lines $\{\lambda_1\} \times \mathbb{R}$ and $\mathbb{R} \times \{\lambda_1\}$ are isolated in $\sum_{p,\mu,\theta}$. 
\end{proof}
Next, we have a topological lemma that is borrowed from the \cite{Cuesta1999fucik}.
\begin{lemma}\label{topology_of_sphere}
Let $S$ be as in \eqref{unit_sphere}; then it has the following properties
\begin{enumerate}
\item $S$ is locally arcwise connected.
\item Any open, connected subset $O$ of $S$ is arcwise connected.
\item If $O'$ is any connected component of an open set $O \subset S$, then $\partial O' \cap O = \emptyset$.
\end{enumerate}
\end{lemma}

\noi The next lemma also follows from the Lemma $3.6$ of \cite{Cuesta1999fucik}
\begin{lemma}\label{components_containing_critical_point}
Let $O = \{u \in S : \mathcal{J}_d(u) < r\}$, then any connected component of $O$ contains a critical point of $\mathcal{J}_d$.
\end{lemma}

\noi The next two lemmas are crucial in estimating the value of the paths considered to prove the existence of the first nontrivial curve in the \fucik spectrum.
\begin{lemma}[\cite{brasco_second_eigenvalue}]\label{crucial_lemma_in_paths}
Let $1 \leq p \leq \infty$ and $U, V \in \mathbb{R}$ such that $U V \leq 0$. Define the following function
\[
g(t) = |U - tV|^p + |U - V|^{p-2} (U - V) V |t|^p, \quad t \in \mathbb{R}.
\]
Then we have
\[
g(t) \leq g(1) = |U - V|^{p-2} (U - V) U, \quad t \in \mathbb{R}.
\]
\end{lemma}
\begin{lemma}[Lemma $4.1$ of \cite{franzina_fractional_p_ev}]\label{crucial_lemma_in_paths_2}
Let $\alpha \in (0,1)$ and $p > 1$. For any nonnegative functions $u, v \in W^{1,p}(\Omega)$, consider the function $\sigma_t(x) := [(1-t)v(x)^p + tu(x)^p]^{\frac{1}{p}}$ for all $t \in [0,1]$. Then for all $t \in [0,1]$,

\begin{align*}
    \int_{\mathbb{R}^N} \int_{\mathbb{R}^N} \frac{|\sigma_t(x) - \sigma_t(y)|^p \, dx \, dy}{|x - y|^{N + ps}} \leq (1-t) \int_{\mathbb{R}^N} \int_{\mathbb{R}^N} \frac{|v(x) - v(y)|^p \, dx \, dy}{|x - y|^{N + ps}} \\
    + t \int_{\mathbb{R}^N} \int_{\mathbb{R}^N} \frac{|u(x) - u(y)|^p  \, dx \, dy}{|x - y|^{N + ps}}.
\end{align*}
\end{lemma}

\begin{theorem}\label{first_non_trivial_curve}
Let $d \geq 0$, then the point $(d + c(d), c(d))$ is the first nontrivial point of $\sum_{p,\mu,\theta}$ in the intersection between $\sum_{p,\mu,\theta}$ and the line $(d,0) + t(1,1)$.
\end{theorem}
\begin{proof}
Assume by contradiction that there exists $\eta$ such that $\lambda_1 < \eta < c(d)$ and $(d + \eta, \eta) \in \sum_{p,\mu,\theta}$. Using the fact that the curves $\{\lambda_1\} \times \mathbb{R}$ and $\mathbb{R} \times \{\lambda_1\}$ are isolated in $\sum_{p,\mu,\theta}$ and $\sum_{p,\mu,\theta}$ is closed, we can choose such a point with $\eta$ minimum. Then $\tilde{\mathcal{J}}_d$ has a critical value $\eta$ corresponding to $u \in S$ with $\lambda_1 < \eta < c(d)$, but there is no critical value in $(\lambda_1, \eta)$. We will construct a path connecting from $\phi_1$ to $-\phi_1$ such that $\tilde{\mathcal{J}}_d \leq \eta$, then we get a contradiction with the definition of $c(d)$, which will complete the proof.

\noi Now, $u$ satisfies weakly
\begin{equation}\label{nontrivial_eq1}
    -\plap u + \pfrac u - \mu \frac{|u|^{p-2}u}{|x|^{p \theta}} = (d + \eta) (u^+)^{p-1} - \eta (u^-)^{p-1}.
\end{equation}
Testing \eqref{nontrivial_eq1} by $u^+$ and $u^-$, we have 
\small{ \begin{align}\label{non_trivial_eq_plus}
\int_{\om} | \grad u^+|^p + \rnnn \frac{F(u(x) - u(y))(u^+(x) - u^+(y))}{|x - y|^{N + p s}} - \mu \int_\Omega \frac{(u^+)^p}{|x|^{p\theta}}   =  (d + \eta) \int_\Omega  (u^+)^p ,
\end{align}}
and
\begin{align}\label{non_trivial_eq_minus}
\int_{\om} | \grad u^-|^p - \rnnn \frac{F(u(x) - u(y))(u^-(x) - u^-(y))}{|x - y|^{N + p s}} - \mu \int_\Omega \frac{(u^-)^p}{|x|^{p\theta}}   =  \eta \int_\Omega  (u^-)^p .
\end{align}
Using \eqref{non_trivial_eq_plus} and \eqref{non_trivial_eq_minus} together with \eqref{brasco_inequality_positive}, and \eqref{brasco_inequality_negative}, we obtain
\[
\int_{\om} | \grad u^+|^p + \rnnn \frac{|u^+(x) - u^+(y)|^p}{|x - y|^{N + p s}} \, dxdy - \mu \int_\Omega \frac{(u^+)^p}{|x|^{p\theta}} \, dx \leq  (d + \eta) \int_\Omega  (u^+)^p \, dx,
\]
and
\[
\int_{\om} | \grad u^-|^p + \rnnn \frac{|u^-(x) - u^-(y)|^p}{|x - y|^{N + p s}} \, dxdy - \mu \int_\Omega \frac{(u^-)^p}{|x|^{p\theta}} \, dx \leq  \eta \int_\Omega  (u^-)^p \, dx,
\]
Thus, we obtain,
\begin{equation}\label{non_trivial_eq3}
\tilde{\mathcal{J}}_d(u) = \eta, \quad \tilde{\mathcal{J}}_d \left( \frac{u^+}{\| u^+ \|_{L^p}} \right) \leq \eta
\end{equation}
\begin{equation}\label{non_trivial_eq4}
\tilde{\mathcal{J}}_d \left( \frac{u^-}{\| u^- \|_{L^p}} \right) \leq \eta - d, \quad \tilde{\mathcal{J}}_d \left( \frac{-u^-}{\| u^- \|_{L^p}} \right) \leq \eta.
\end{equation}
Now, we define three paths $\gamma_1, \gamma_2,$ and $\gamma_3$ in $S$, which go from $u$ to $\frac{u^+}{\|u^+\|_{L^p}}$, $\frac{u^+}{\|u^+\|_{L^p}}$ to $\frac{u^-}{\|u^-\|_{L^p}}$, and $-\frac{u^-}{\|u^-\|_{L^p}}$ to $u$ respectively.

\[
\gamma_1(t) = \frac{u^+ - (1 - t) u^-}{\|u^+ - (1 - t) u^-\|_{L^p}}, \quad \gamma_2(t) = \frac{[(1 - t) (u^+)^p + t (u^-)^p]^{1/p}}{\|(1 - t) (u^+)^p + t (u^-)^p\|_{L^p}}, \quad \gamma_3(t) = \frac{t u^+ -  u^-}{\|t u^+ -  u^-\|_{L^p}}.
\]

\noi Using \eqref{non_trivial_eq_plus}, \eqref{non_trivial_eq_minus} and Lemma \ref{crucial_lemma_in_paths} with $U = u^+(x) - u^+(y)$ and $V = u^-(x) - u^-(y)$, we deduce that for all $t \in [0, 1]$,
%%% first curve
\begin{align*}
\tilde{\mathcal{J}}(\gamma_1(t)) & \|u^+ - (1 - t) u^-\|_{L^p}^p = \left[\int_\Omega |\nabla u^+|^p \, dx  - \mu \int_{\om} \frac{|u^+|^p}{|x|^{p \theta}} ~ dx - d \int_{\om} (u^+)^p \right] \\
&+ (1 - t)^p \left[\int_\Omega |\nabla u^-|^p \, dx  - \mu \int_{\om} \frac{|u^-|^p}{|x|^{p \theta}} ~ dx  \right]  +  \int_{\mathbb{R}^N} \int_{\mathbb{R}^N} \frac{|U - (1-t)V|^{p} }{|x - y|^{N + p s}} \, dxdy\\
= &  \left[-\rnnn \frac{|U - V|^{p-2}(U - V)(U)}{|x - y|^{N + p s}} \,  + \eta \int_{\om} (u^+)^p \right]\\
& + (1 - t)^p \left[  \rnnn \frac{|U - V|^{p-2}(U - V)V }{|x - y|^{N + p s}} + \eta \int_{\om} (u^-)^p  \right] +   \int_{\mathbb{R}^N} \int_{\mathbb{R}^N} \frac{|U - (1-t)V|^{p} }{|x - y|^{N + p s}}  \\
\leq & ~\eta \int_{\om} (u^+)^p + \eta ( 1- t)^p \int_{\om} (u^-)^p  .
\end{align*}

\noi Taking $\sigma_t(x) = [(1 - t) (u^+)^p + t (u^-)^p]^{1/p}(x)$, then by means of Lemma \ref{crucial_lemma_in_paths_2}, we deduce
%%% second curve
\begin{align*}
    \tilde{\mathcal{J}}(\gamma_2(t)) & \|(1 - t) (u^+)^p + t (u^-)^p\|_{L^p}^p =   (1 - t) \left[ \int_\Omega |\nabla u^+|^p \, dx - \mu  \int_{\om} \frac{|u^+|^p}{|x|^{p \theta}} ~ dx -  d \int_{\om} (u^+)p \right]   \\
    &~   +t \left[  \int_\Omega |\nabla u^-|^p \, dx -   \mu \int_{\om} \frac{|u^-|^p}{|x|^{p \theta}} ~ dx \right] + \rnnn \frac{\sigma_t(x) - \sigma_t(y)}{|x - y|^{N + ps}} \\
    \leq & ~ (1-t) \left[ \int_\Omega |\nabla u^+|^p \, dx  - \mu \int_{\om} \frac{|u^+|^p}{|x|^{p \theta}} ~ dx   -  d \int_{\om} (u^+)^p + \rnnn \frac{|U|^p}{|x-y|^{N + ps}} \right] \\
     & ~ +t \left[  \int_\Omega |\nabla u^-|^p \, dx -   \mu \int_{\om} \frac{|u^-|^p}{|x|^{p \theta}} ~ dx + \rnnn \frac{|V|^p}{|x-y|^{N+ps}} \right] \\
     \leq & ~  (1-t) \tilde{\mathcal{J}}(u^+) + t  \tilde{\mathcal{J}}(u^-) \\
     \leq &~ \eta \|(1 - t) (u^+)^p + t (u^-)^p\|_{L^p}^p.
\end{align*}
Again from \eqref{non_trivial_eq_plus}, \eqref{non_trivial_eq_minus} and Lemma \ref{crucial_lemma_in_paths} with $U = u^+(x) - u^+(y)$ and $V = u^-(x) - u^-(y)$, we obtain
\begin{align*}
\tilde{\mathcal{J}} (\gamma_3(t)) & \|u^+ - (1 - t) u^-\|_{L^p}^p = t^p \left[ \int_\Omega |\nabla u^+|^p \, dx  - \mu \int_{\om} \frac{|u^+|^p}{|x|^{p \theta}} ~ dx - d \int_{\om} (u^+)p \right] \\
&+  \left[ \int_\Omega |\nabla u^-|^p \, dx  - \mu \int_{\om} \frac{|u^-|^p}{|x|^{p \theta}} ~ dx \right] +  \int_{\mathbb{R}^N} \int_{\mathbb{R}^N} \frac{|t U - V|^{p} }{|x - y|^{N + p s}} \, dxdy \\
=& ~  t^p \left[ \rnnn \frac{|V - U|^{p-2}(V - U)U}{|x - y|^{N + p s}} \, dxdy + \eta \int_{\om} (u^+)^p \right] \\
&+  \left[  \rnnn \frac{-| V - U|^{p-2}(V - U)V }{|x - y|^{N + p s}} \, dxdy + \eta \int_{\om} (u^-)^p  \right] +  \int_{\mathbb{R}^N} \int_{\mathbb{R}^N} \frac{|V - t U|^{p} }{|x - y|^{N + p s}} \, dxdy\\
\leq & ~ \eta \int_{\om} (u^+)^p + ( 1- t)^p \eta \int_{\om} (u^-)^p .
\end{align*}
Thus, with the help of the paths $\gamma_1, \gamma_2$ and $\gamma_3$, we move from $-\frac{u^-}{\lv u^- \rv_{L^p}}$ to $\frac{u^-}{\lv u^- \rv_{L^p}}$ while keeping $\tilde{\mathcal{J}}(\gamma_i(t)) \leq \eta$.\\
Next, we investigate the levels below $\eta - d$. For this, consider $\mathcal{O} = \{ v \in S : \tilde{\mathcal{J}}_d(v) < \eta - d \}$. From \eqref{non_trivial_eq3} and \eqref{non_trivial_eq4}, we have $\phi_1 \in \mathcal{O}$ and $- \phi_1 \in \mathcal{O}$ if $\lambda_1 < \eta - d$. Additionally, from the choice of $\eta$, $\phi_1$ and $-\phi_1$ are the only possible critical points of $\tilde{\mathcal{J}}_d$ in $\mathcal{O}$.  Also, $\frac{u^-}{\|u^-\|_{L^p}}$ is not a critical point of $\tilde{\mathcal{J}}_d$, thanks to the fact that $\frac{u^-}{\|u^-\|_{L^p}}$ does not change sign and vanishes on a set of positive measure. Therefore, there exists a $C^1$ path $\sigma : [-\e, \e] \to \mathcal{S}$ with $\sigma(0) = \frac{u^-}{\|u^-\|_{L^p}}$ and $\frac{d}{dt} \tilde{\mathcal{J}}_d(\sigma(t))|_{t=0} \neq 0$.

With the help of this path, we can move from $\frac{u^-}{\|u^-\|_{L^p}}$ to a point $w$ with $\tilde{\mathcal{J}}_d(w) < \eta - d $. Applying Lemma \ref{components_containing_critical_point} along with Lemma \ref{topology_of_sphere}, we can continue from $w$ to $\phi_1$ or $- \phi_1$ with a path, say $\tilde{\sigma}$, in $S$ with levels strictly below $\eta - d$. Let us assume that it is $\phi_1$. At this point, we construct a path $\gamma_4(t)$ from $\frac{u^-}{\|u^-\|_{L^p}}$ to $\phi_1$. \\
Observe that $$ | \tilde{\mathcal{J}}_d(v) - \tilde{\mathcal{J}}_d(-v) | \leq d.$$
Then $$ \tilde{\mathcal{J}}_d( - \gamma_4(t)) \leq \tilde{\mathcal{J}}_d(\gamma_4(t) + d \leq \eta.$$
Thus, on the path $-\gamma_4(t)$, we are able to move from $-\phi_1$ to $- \frac{u^-}{\lv u^- \rv_{L^p}}$ by staying at the levels less than or equal to $\eta$. \\
Putting all together, we have constructed a path from $-\phi_1$ to $\phi_1$ while staying at levels less than $\eta$. This leads to a contradiction to the choice of $\eta$. Hence, we have our result.
\end{proof}

\begin{proposition}
The curve $d \mapsto (d + c(d), c(d))$, $d \in \mathbb{R}^+$ is nonincreasing (in the sense that $d_1 < d_2$ implies $d_1 + c(d_1) \leq d_2 + c(d_2)$ and $c(d_1) \geq c(d_2)$) and Lipschitz continuous.
\end{proposition}
\begin{proof}
To prove Lipschitz continuity, take $d < d'$ and observe that $\tilde{\mathcal{J}}_d(u) \geq \tilde{\mathcal{J}}_{d'}(u)$. This implies that $c(d) \geq c(d')$. By definition of $c(d')$, for any $\e > 0$, there exists a curve $\gamma \in \Gamma$ such that $$ \max_{u \in \gamma[-1,1]} \tilde{\mathcal{J}}_{d'}(u) \leq c(d') + \e.$$
Additionally, there exists $u_0 \in S$ such that $\tilde{\mathcal{J}}_d(u_0) = \max_{u \in \gamma[-1,1]} \tilde{\mathcal{J}}_{d}(u)$. Combining, we obtain 
\begin{align*}
0 \leq c(d) - c(d') \leq \max_{u \in \gamma[-1,1]} \tilde{\mathcal{J}}_{d}(u) - \tilde{\mathcal{J}}_{d'} (u_0) + \e \leq \tilde{\mathcal{J}}_d(u_0) - \tilde{\mathcal{J}}_{d'} (u_0) + \e 
\leq  d' - d + \e.
\end{align*}
Since, $\e > 0$ is arbitrary, the map is Lipschitz continuous. 
% \textcolor{red}{Proof of this requires} Lemma \ref{non_trivial_can't_vary}.
\end{proof}
\noi As $c(d)$ is decreasing and positive, so the limit of $c(d)$ exists as $d \to \infty$. The exact value of the limit is nothing but the value of the first eigenvalue only.

\begin{theorem}
Let $1 < p < \infty$. Then $\ds \lim_{d \to \infty} c(d) = \lambda_1$.
\end{theorem}
\begin{proof}
The proof follows in a similar manner as in Proposition $4.4$ of \cite{Cuesta1999fucik}. For the sake of completeness, we provide the proof below.

\noi On the contrary, assume there exists $\delta > 0$ such that

\begin{equation} \label{eq:maxJd}
    \max_{u \in \, \gamma \in [ -1, 1 ]} \tilde{\mathcal{J}}_d (u) \geq \lambda_1 + \delta
\end{equation}
for all $\gamma \in \Gamma$ and for all $d \geq 0$.
Choose $\phi \in \sobo(\om)$, hence in $X(\om)$, as in Lemma $4.3$ of \cite{Cuesta1999fucik}. Consider the function
\begin{equation*} 
    \gamma(t) = \frac{t \phi_1 + (1 - |t|) \phi}{\| t \phi_1 + (1 - |t|) \phi \|_p}, \quad t \in [-1,1].
\end{equation*}
Since $\tilde{\mathcal{J}}_d$ is continuous, it attains its maximum at some $t = t_d$. Defining $v_{t_d} = t_d \phi_1 + (1 - |t_d|) \phi, $ we obtain the inequality
\begin{equation*} 
    \int_\Omega |\nabla v_{t_d}|^p + \int_\Omega |v_{t_d}|^p - \mu \int_\Omega \frac{|v_{t_d}|^p}{|x|^{p \theta}} - d \int_\Omega (v_{t_d}^+)^p 
    \geq (\lambda_1 + \delta) \int_\Omega |t_d \phi_1 + (1 - |t_d|) \phi|^p.
\end{equation*}
Letting $d \to \infty$, we extract a subsequence such that $t_d \to \bar{t} \in [-1,1]$. This forces $\int_\Omega (v_{t_d}^+)^p \to 0.$ Thus, we obtain
\begin{equation*}
    \int_\Omega \left[ \big(\bar{t} \phi_1 + (1 - |\bar{t}|) \phi \big)^+\right]^p = 0.
\end{equation*}
This implies that
\begin{equation*} 
    \bar{t} \phi_1 + (1 - |\bar{t}|) \phi \leq 0 \quad \text{a.e. in } \Omega.
\end{equation*}
Consequently, we must have $\bar{t} = -1$.
But then \eqref{eq:maxJd} implies that 
\begin{equation*}
    \lambda_1 \int_\Omega |\phi_1|^p = \int_\Omega |\nabla \phi_1|^p +[\phi_1]_s^p - \mu \int_\Omega \frac{|\phi_1|^p}{|x|^{p \theta}} \geq (\lambda_1 + \delta) \int_\Omega |\phi_1|^p.
\end{equation*}
This gives us a contradiction and hence completes the proof.
\end{proof}

%%%%%%%%%%%%%%%%%%%%%%%%%%%%%%%%%%%%%%
%%%%%%%% Characterisation of 2nd eigenvalue
%%%%%%%%%%%%%%%%%%%%
%%%%%%%%%%%%%%%%%%%%%%%%%%%%%%%%%%%%%%
%%%%%%%% Regularity of Eigenfunctions
%%%%%%%%%%%%%%%%%%%%
\section{Shape Optimization and Regularity}
%%%%%%%%%%%%%%%%%%%%%%%%%%%%%%%%
%%%%% shape optimization problem of first eigenvalue
%%%%%%%%%%%%%%%%%%%%%%%%%%%%%%%
In this section, we provide proofs of the shape optimization problems of the domains corresponding to the first two eigenvalues. We start with proving one basic property of first eigenvalue of the operator $\mathcal{T}$.
\begin{theorem}
Let $\Omega_1, \om_2 \subset \rnn$ be domains such that $\Omega_1$ is a proper subset of $\om_2$. Then, $\lambda_1(\Omega_2) < \lambda_1(\Omega_1)$.
\end{theorem}

\begin{proof}
Let $\phi_1 \in X(\Omega_1)$ be an eigenfunction associated to $\lambda_1(\Omega_1)$, and define $\tilde{\phi_1}$ as the function obtained by extending $\phi_1$ by 0 in $\Omega_2 \setminus \Omega_1$. Then, $\tilde{\phi_1} \in X(\Omega_2)$, and
\[
\int_{\Omega_2}  |\tilde{\phi_1}|^p \, dx = \int_{\Omega_1}  |\phi_1|^p \, dx > 0.
\]
Using $\frac{\tilde{\phi_1}}{\left(\int_{\Omega_2}  |\tilde{\phi_1}|^p \right)^{1/p}} $ as an admissible function for $\lambda_1(\Omega_2)$, we obtain
\begin{align*}
\lambda_1(\Omega_2) & \leq   \frac{\int_{\om_2} |\grad \tilde{\phi_1}|^p + \rnnn \frac{|\tilde{\phi_1}(x) - \tilde{\phi_1}(y)|^p}{|x - y|^{N + ps}} \, dx \, dy - \mu \int_{\Omega_2} \frac{|\tilde{\phi_1}|^p}{|x|^{p\theta}} \, dx}{\int_{\Omega_2}  |\tilde{\phi_1}|^p \, dx} \\
& =  \frac{\int_{\om_1} |\grad \phi_1|^p + \rnnn \frac{|\phi_1(x) - \phi_1(y)|^p}{|x - y|^{N + ps}} \, dx \, dy - \mu \int_{\Omega_1} \frac{|\phi_1|^p}{|x|^{p\theta}} \, dx}{\int_{\Omega_1}  |\phi_1|^p \, dx} \\
& =  \lambda_1(\Omega_1).
\end{align*}
The equality holds only if $\tilde{\phi_1}$ is an eigenfunction associated with $\lambda_1(\Omega_2)$, but this is impossible because $|\tilde{\phi_1} = 0| > 0$ leads to a contradiction. Hence, $~\lambda_1(\Omega_2) < \lambda_1(\Omega_1)$.
\end{proof}

%%%%%%%%%% Faber Krahn Inequality
\noi Finally, we give the answer to the optimization problem asked in the setup of the operator $\mathcal{T}$, which says that in the class of all domains with fixed volume $c$, the ball has the smallest first eigenvalue.
\begin{theorem}[Faber-Krahn inequality]
    Let $1 < p < N$, $c$ be a positive real number, and $B$ be a ball with Lebesgue measure $c$ in $\rnn$. Then $$ \la_1(B) = \inf \{ \la_1(\om) : \om \mbox{ be a domain in } \rnn, |\om| = c \} .$$
\end{theorem}
\begin{proof}
    Let $\om$ be a domain of Lebesgue measure $c$ and $\om^*$ be the ball of the same measure centred at $0$. Additionally, take $\phi_1$ be the eigenfunction corresponding to the first eigenvalue $\la_1(\om)$ and $\phi_1^*$ be its corresponding Schwarz Symmetrization. \\
\noi Now, take $W_k(x) = \frac{\mathds{1}_{\{|z|>1/k\}}(x)}{|x|^{N +ps}} $ where $\mathds{1}_E$ denotes the characteristic function of a set $E$. Also, there exists a sequence $\psi_m \in C_c^{\infty}(\om)$ with $\psi_m \geq 0$ such that $\lv \psi_m - \phi_1\rv_{L^p(\om)} \to 0$ as $m \to \infty$. Theorem $3.5$ of \cite{lieb2001analysis} implies that $\lv \psi^*_m - \phi_1^*\rv_{L^p(\om)} \to 0$ as $k \to \infty$. Now, Corollary $2.3$ of \cite{almgren_symmetric_decreasing} implies that
\begin{equation}\label{eq:faber1}
    \rnnn \frac{ \mathds{1}_{\{ |z| > 1/k \} }(x-y)|\psi_m^*(x) -\psi_m^*(y)|^p}{|x-y|^{N+ps}}  \leq  \rnnn \frac{ \mathds{1}_{\{ |z| > 1/k \} }(x-y)|\psi_m(x) -\psi_m(y)|^p}{|x-y|^{N+ps}} .
\end{equation}
Since $|a +b |^p \leq 2^{p-1}(|a|^p + |b|^p)$, using Fubini's Theorem and change of variable, we obtain
\begin{equation*}
    \rnnn \frac{ \mathds{1}_{\{ |z| > 1/k \} }(x-y)|\psi_m(x) -\psi_m(y)|^p}{|x-y|^{N+ps}} dy~ dx \leq 2^{p} \int_{\rnn} \frac{\mathds{1}_{\{ |z| > 1/k \} }(y)}{|y|^{N+ps}} dy \int_{\rnn} |\psi_m(x)|^pdx. 
\end{equation*}
\noi Applying Fatou's Lemma and Dominated Convergence Theorem in \eqref{eq:faber1}, we deduce
\begin{equation*}
     \rnnn \frac{ \mathds{1}_{\{ |z| > 1/k \} }(x-y)|\phi_1^*(x) -\phi_1^*(y)|^p}{|x-y|^{N+ps}} \leq  \rnnn \frac{ \mathds{1}_{\{ |z| > 1/k \} }(x-y)|\phi_1(x) -\phi_1(y)|^p}{|x-y|^{N+ps}} .
\end{equation*}
Finally, using the Monotone Convergence Theorem, we have
\begin{equation}\label{eq:faber2}
    \rnnn \frac{ |\phi_1^*(x) -\phi_1^*(y)|^p}{|x-y|^{N+ps}} dy~ dx \leq  \rnnn \frac{ |\phi_1(x) -\phi_1(y)|^p}{|x-y|^{N+ps}} dy~ dx.
\end{equation}
Also, using the Rearrangement inequality (see Theorem $3.4$ and property $(v)$ in section $3.3$ of \cite{lieb2001analysis}), we obtain
\begin{equation}\label{eq:faber3}
\int_{\om} \frac{|\phi_1|^p}{|x|^p}~dx \leq \int_{\om^*} (|\phi_1|^p)^* \left(\frac{1}{|x|^p}\right)^* dx = \int_{\om^*} \frac{(|\phi_1|^p)^*}{|x|^p} dx = \int_{\om^*} \frac{(\phi_1^*)^p}{|x|^p} dx
\end{equation}
\noi Next, as we know that $\int_{\om} \phi_1^p = \int_{\om^*} (\phi_1^*)^p$. Then, using P\`{o}lya's inequality (see Theorem $2.1.3$ of \cite{henrot_extremum_problems}), \eqref{eq:faber2}, and \eqref{eq:faber3}, we get  
\begin{align*}
    \la_1(\om^*) & \leq \frac{\int_{\om^*} |\grad \phi_1^*|^p dx + \rnnn \frac{|\phi_1^*(x) - \phi_1^*(y)|^p}{|x-y|^{N+ps}}~dxdy - \mu \int_{\om^*} \frac{|\phi_1^*|^p}{|x|^p}~dx  }{\int_{\om^*} (\phi_1^*)^p ~dx}\\
    & \leq \frac{\int_{\om} |\grad \phi_1|^p dx + \rnnn \frac{|\phi_1(x) - \phi_1(y)|^p}{|x-y|^{N+ps}}~dxdy - \mu \int_{\om} \frac{|\phi_1|^p}{|x|^p}~dx  }{\int_{\om} (\phi_1)^p ~dx} = \la_1(\om).
\end{align*}
This proves our assertion. 
\end{proof}

%%%%%%%% Regularity of Eigenfunctions
We start this section by obtaining some regularity estimates of the eigenfunctions, which are essential in proving further properties of the spectrum of the operator $\mathcal{T}$.
\begin{theorem}\label{regularity_of_eigenfunctions}
Let $1 < p < N$ and $\phi$ be an eigenfunction associated with the eigenvalue $\la$ of the operator $\mathcal{T}$. Then $\phi \in L^{\infty}_{loc}(\om\setminus\{ 0 \})$. Furthermore, for $p \geq 2$ there exists some $\delta > 0$ such that $\phi \in C_{loc}^{\delta}(\om\setminus \{ 0 \})$. In particular, $\phi$ is continuous on $\om \setminus \{ 0 \}$. 
\end{theorem}
\begin{proof}
Choose a ball $B_R(x_0)$ such that $d(B_R(x_0),0) \geq \rho$. Then for all $v$ such that $\phi v \geq 0$ with $Supp(v) \subset B_R(x_0)$, we obtain 
\begin{align}\label{local_bdd_test_eqn}
\int_{\om} |\grad \phi|^{p-2} \grad \phi \grad v + \rnnn \frac{F(\phi(x) - \phi(y))(v(x) - v(y))}{|x - y|^{N+ps}} & = \mu \int_{\om} \frac{|\phi|^{p-2} \phi v}{|x|^{p \theta}} + \la \int_{\om} |\phi|^{p-2} \phi v \notag \\
& \leq \la_{\rho} \int_{\om} |\phi|^{p-2}\phi v,
\end{align}
where $\la_{\rho} = \frac{\mu}{\rho^{p \theta}} + \la $.
\noi Taking $v(x) =  \zeta^p(x) \phi_+^{\beta - (p-1)}(x)$ as the test function with $\al = \beta - (p-1)$ and $\zeta \in C_c^{1}(\om)$ such that $0 \leq \zeta \leq 1$, $\zeta \equiv 1$ in $B_{r^*}(x_0)$, $Supp(\zeta) \subset B_{R_*}(x_0)$ with $r^* < R^* \leq R$ and $|\grad \zeta | \leq \frac{1}{R^*-r^*}$. We have
\begin{equation}\label{local_bdd_test_derivative}
    \grad v = p \zeta^{p-1} \phi_+^{\al} \grad \zeta + \zeta^p \al \phi_+^{\al - 1} \grad \phi_+.
\end{equation}
Also, using the inequality \eqref{fractional_minimum_inequality}, we have 
\begin{equation}\label{local_bdd_frac_non_neg}
\begin{split}
|\phi(x) &- \phi(y)|^{p-2}(\phi(x) - \phi(y))(v(x) - v(y)) \\
& = \begin{cases}
    0 & \mbox{ if } \phi(x) < 0, \phi(y) < 0,\\
    |\phi^-(x) + \phi^+(y)|^{p-2}(\phi^-(x) + \phi^+(y))v(y) & \mbox{ if }\phi(x) < 0, \phi(y) \geq 0,\\
    |\phi^+(x) +\phi^+(y)|^{p-2} (\phi^+(x) + \phi^+(y)) v(x) & \mbox{ if } \phi(x) \geq 0, \phi(y) < 0,\\
    F(\phi(x) - \phi(y))( (\zeta^{p/\al} \phi_+)^{\al}(x) - (\zeta^{p/\al} \phi_+)^{\al}(y)) & \mbox{ if } \phi(x) \geq 0, \phi(y) \geq 0,
\end{cases}
\\ & \geq 0.
\end{split}
\end{equation}

\noi Using \eqref{local_bdd_test_derivative} and \eqref{local_bdd_frac_non_neg} in \eqref{local_bdd_test_eqn}, we deduce 
\begin{align*}
\al \int_{\om} |\grad \phi_{+}|^p \zeta^p \phi_{+}^{\al-1} \leq - p \int_{\om} \zeta^{p-1} \phi_{+}^{\al} \langle |\grad \phi_{+}|^{p-2} \grad \phi_{+}, \grad \zeta\rangle + \la_{\rho} \int_{\om} \phi_{+}^{p-1+\al}  \zeta^p.
\end{align*}
As $ \al = \frac{(\al - 1)(p-1)}{p} + \frac{\al + p -1}{p}$, the Young's inequality yields
\begin{align*}
\al \int_{\om} |\grad \phi_{+}|^p \zeta^p \phi_{+}^{\al-1} &\leq p \int_{\om} \left( \zeta^{p-1} |\grad \phi_{+} |^{p-1} \phi_{+}^{\frac{(\al - 1)(p-1)}{p}}\right) \left( \phi_{+}^{\frac{\al + p -1}{p}} |\grad \zeta|\right) + \la_{\rho} \int_{\om} \zeta^p \phi_{+}^{p-1+\al}\\
&\leq  p \int_{\om} \left[ \frac{\al \e^{\frac{p}{p-1}} \zeta^p |\grad \phi_{+}|^{p} \phi_{+} ^{\al -1}}{p/(p-1)} + \frac{\phi_{+}^{\al + p -1} |\grad \zeta|^p}{p \e^p \al^{p-1} }\right] + \la_{\rho} \int_{\om} \zeta^p \phi_{+}^{p-1+\al}
\end{align*}
for some $\e > 0$. This implies
\begin{equation*}
    \al \left( 1 - (p-1) \e^{\frac{p}{p-1}}\right) \int_{\om} |\grad \phi_{+}|^{p} \zeta^p \phi_{+}^{\al - 1} \leq \frac{1}{\e^p \al^{p-1}} \int_{\om} \phi_{+}^{\al + p -1} |\grad \zeta |^p + \la_{\rho} \int_{\om} \zeta^p \phi_{+}^{p-1+\al}.
\end{equation*}
Choosing $\e$ such that $(p-1) \e^{\frac{p}{p-1}} \leq \frac{1}{2}$, we have
\begin{equation*}
    \int_{\om} |\grad \phi_{+}|^{p} \zeta^p \phi_{+}^{\al - 1} \leq \frac{2}{\e^p \al^{p}} \int_{\om} \phi_{+}^{\al + p -1} |\grad \zeta |^p + \frac{2 \la_{\rho}}{\al} \int_{\om} \zeta^p \phi_{+}^{p-1+\al},
\end{equation*}
that can be rewritten in terms of $\beta$ as 
\begin{equation*}
    \int_{\om} \zeta^p |\grad \phi_{+}^{\frac{\beta}{p}}|^p \leq \frac{2 \beta^p}{(\e p(\beta - p +1))^p} \int_{\om} \phi_{+}^{\beta} |\grad \zeta|^p + \frac{2 \la_{\rho} \beta^{p}}{(\beta - p +1) p^p} \int_{\om} \zeta^p \phi_{+}^{\beta}.
\end{equation*}
Thus, Sobolev inequality gives 
\begin{align*}
    S \left( \int_{\om} |\zeta \phi_{+}^{\beta/p}|^{p^*} \right)^{\frac{p}{p^*}} & \leq \int_{\om} | \grad ( \zeta \phi_{+}^{\beta/p})|^p \\
    & \leq 2^{p-1} \int_{\om} \zeta^p |\grad \phi_{+}^{\frac{\beta}{p}}|^p + 2^{p-1} \int_{\om} |\grad \zeta|^p |\phi_{+}|^{\beta}\\
    & \leq \left( 1 +  \frac{2 \beta^p}{(\e p(\beta - p +1))^p} \right) 2^{p-1} \int_{\om} \phi_{+}^{\beta} |\grad \zeta|^p + \frac{2^p \la_{\rho} \beta^{p}}{(\beta - p +1) p^p} \int_{\om} \zeta^p \phi_{+}^{\beta}.
\end{align*}
Using the properties of $\zeta$ and $\beta \geq p$, we deduce
\begin{equation*}
\begin{split}
\lv \phi_{+} \rv_{L^{\beta p^*/p}(B_{r^*})} & \leq (S^{-1})^{1/\beta} \beta^{p/\beta}\left[  \left( \frac{1}{\beta^p} +  \frac{2 }{(\e p(\beta - p +1))^p} \right) \frac{2^{p-1}}{(R^* -  r^*)^p}  \right. \\
& \quad \left. + \frac{2^p \la_{\rho}}{(\beta - p +1) p^p} \right]^{1/\beta} \lv \phi_{+} \rv_{L^{\beta}(B_{R^*})}\\
&\leq (S^{-1})^{1/\beta} \beta^{p/\beta}\left[  \left( \frac{1}{p^p} +  \frac{2 }{(\e p)^p} \right) \frac{2^{p-1}}{(R^* - r^*)^p} + \frac{2^p \la_{\rho}}{p^p} \right]^{1/\beta} \lv \phi_{+} \rv_{L^{\beta}(B_{R^*})}.
\end{split}
\end{equation*}
Taking $c_0 = \left( \frac{1}{p^p} +  \frac{2 }{(\e p)^p} \right), \chi = \frac{p^*}{p}, \gamma_i = \beta_0 \chi^i$ with $\beta_0 = p$, and $r_j = r + 2^{-j}(R -r)$ for $j \in \ntrl \cup \{0\}$. Then 
\begin{equation*}
\begin{split}
    \lv \phi_{+} \rv_{L^{\gamma_1}(B_{r_1})} & \leq (S^{-1})^{\frac{1}{\gamma_0}} \gamma_0^{\frac{p}{\gamma_0}}\left[  \frac{2^{p-1}c_0 }{(r_0 - r_1)^p} + \frac{2^p \la_{\rho}}{p^p} \right]^{\frac{1}{\gamma_0}} \lv \phi_{+} \rv_{L^{\gamma_0}(B_{r_0})}\\
    & \leq (2^{p-1} S^{-1})^{\frac{1}{\gamma_0}}\gamma_0^{\frac{p}{\gamma_0}}\left[  \frac{2^p c_0 }{(R - r)^p} + \frac{2 \la_{\rho}}{p^p} \right]^{\frac{1}{\gamma_0}} \lv \phi_{+} \rv_{L^{\gamma_0}(B_{r_0})}.
\end{split}
\end{equation*}
Thus, we have 
\begin{equation*}
    \begin{split}
         &\lv \phi_{+} \rv_{L^{\gamma_2}(B_{r_2})} \leq (2^{p-1} S^{-1})^{\frac{1}{\gamma_1}}\gamma_1^{\frac{p}{\gamma_1}}\left[  \frac{2^{2p} c_0 }{(R - r)^p} + \frac{2 \la_{\rho}}{p^p} \right]^{1/\gamma_1} \lv \phi_{+} \rv_{L^{\gamma_1}(B_{r_1})}\\
         &\leq (2^{p-1} S^{-1})^{\frac{1}{\gamma_1} + \frac{1}{\gamma_0}} \gamma_1^{\frac{p}{\gamma_1}} \gamma_0^{\frac{p}{\gamma_0}} \left[  \frac{2^{2p} c_0 }{(R - r)^p} + \frac{2 \la_{\rho}}{p^p} \right]^{\frac{1}{\gamma_1}} \left[  \frac{2^p c_0 }{(R - r)^p} + \frac{2 \la_{\rho}}{p^p} \right]^{\frac{1}{\gamma_0}} \lv \phi_{+} \rv_{L^{\gamma_0}(B_{r_0})}.
    \end{split}
\end{equation*}
Hence, recursively one obtain
\begin{equation*}
     \lv \phi_{+} \rv_{L^{\gamma_j}(B_{r_j})} \leq (2^{p-1}S^{-1})^{\sum_{i=0}^{j-1}\frac{1}{\gamma_i}} \left( \prod_{i=0}^{j-1} \gamma_i^{\frac{p}{\gamma_i}} \left[ \frac{2^{(i+1)p} c_0 }{(R - r)^p} + \frac{2 \la_{\rho}}{p^p} \right]^{\frac{1}{\gamma_i}} \right) \lv \phi_{+} \rv_{L^{\gamma_0}(B_{r_0})},
\end{equation*}
for all $j \geq 1$. As $2^{(i+1)p} \to \infty $ as $i \to \infty$, there exists $c_0^*$ such that $$\frac{2 \la_{\rho}}{p^p} \leq c_0^*\frac{2^{(i+1)p} c_0 }{(R - r)^p} \mbox{  for all } i \geq 1. $$ Hence, 
\begin{equation}\label{local_bdd_final_iteration}
    \lv \phi_{+} \rv_{L^{\gamma_j}(B_{r_j})} \leq  \left( \frac{S^{-1} 2^{p-1}(1+c_0^*)c_0}{(R-r)^p}\right)^{\sum_{i=0}^{j-1} \frac{1}{\gamma_i}} 2^{\sum_{i=0}^{j-1} \frac{(i+1)p}{\gamma_i}} \left( \prod_{i=0}^{j-1} \gamma_i^{\frac{1}{\gamma_i}}  \right)^p \lv \phi_{+} \rv_{L^{\gamma_0}(B_{r_0})}.
\end{equation}
Since 
\begin{align*}
    \sum_{i =0 }^{j-1} \frac{1}{\gamma_i} &= \frac{1}{\beta_0} \sum_{i=0}^{j-1} \frac{1}{\chi^i} \leq \frac{\chi}{\beta_0(\chi - 1)} = \frac{N}{\beta_0 p}\\
     \sum_{i =0 }^{j-1} \frac{i}{\gamma_i} &= \frac{1}{\beta_0} \sum_{i=0}^{j-1} \frac{i}{\chi^i} \leq \frac{\chi}{\beta_0(1-\chi)^2} = \frac{N^2}{\beta_0 p^2}\\
     \sum_{i=0}^{j-1} \frac{\log(\gamma_i)}{\gamma_i} & = \frac{\log(\beta_0)}{\beta_0} \sum_{i=0}^{j-1} \frac{1}{\chi^i} + \frac{\log(\chi)}{\beta_0} \sum_{i=0}^{j-1}\frac{i}{\chi^i} < \infty.
\end{align*}
Therefore, taking $j \to \infty$ in \eqref{local_bdd_final_iteration}, we get
\begin{equation*}
\sup_{B_{r}} |\phi_{+}| \leq C \lv \phi_{+} \rv_{L^{\beta}(B_R)}.
\end{equation*}
Hence, we have the local boundedness of $\phi_+$. Note that, if $\phi$ is an eigenfunction, then so is $-\phi$. Additionally, $(-\phi)_+ = \phi_-$. Applying the same method, we have local boundedness of $\phi_-$. Hence, $\phi$ is locally bounded. \\
The continuity of solution comes from the Theorem $1.4$ of \cite{garain_higher_Holder_regularity}.
\end{proof}

Now, we provide a variational characterization of the second eigenvalue of mixed interpolated Hardy operator $\mathcal{T}$ and prove the mixed Hong-Krahn-Szeg\"{o} inequality.
%%%%%%%% Nodal Domains  
\begin{theorem}[Nodal Domains]\label{thm:nodal_domains}
Let $2 \leq p < N$ and  $\la > \la_1(\om)$ be an eigenvalue of the mixed interpolated Hardy operator $\mathcal{T}$ and $\phi_{\la}$ be the associated eigenfunction. Then we have $$\la > \max\{ \la_1(\om^+), \la_1(\om^-)\}$$ where $\om^+ := \{ x \in \om \setminus \{0\}: \phi_{\la}>0\}$ and $\om^- := \{ x \in \om \setminus \{0\} : \phi_{\la}<0\}$.
\end{theorem}
\begin{proof}
As $\phi_{\la}$ is continuous on $\om \setminus \{0\}$ from Theorem \ref{regularity_of_eigenfunctions}, we have $\om_+ $ and $\om_-$ to be open in $\om \setminus \{0\}$. But $\om \setminus \{0\}$ is open in $\om$ and $\la_1(\om) = \la_1(\om \setminus \{0\})$. This allows us to define the eigenvalues $\la_1(\om^+)$ and $\la_1(\om^-)$. Since $\phi_{\la}$ is an eigenfunction, it satisfies 
\small{
\begin{equation*}
    \int_{\om} |\grad \phi_{\la}|^{p-2} \grad \phi_{\la} \grad v + \rnnn \frac{F(\phi_{\la}(x) - \phi_{\la}(y))(v(x) - v(y))}{|x - y|^{N+ps}} - \mu \int_{\om} \frac{|\phi_{\la}|^{p-2} \phi_{\la} v}{|x|^{p \theta}} = \la \int_{\om} |\phi_{\la}|^{p-2} \phi_{\la} v 
\end{equation*}}
for all $ v \in X(\om)$. \\
Taking $v = \phi_{\la}^+$ and using \eqref{brasco_inequality_positive}, we obtain
% $$( \phi_{\la}(x) - \phi_{\la}(y))(\phi_{\la}^+(x) - \phi_{\la}^+(y)) = (\phi_{\la}^+(x) - \phi_{\la}^+(y))^2 + \phi_{\la}^+(x)\phi_{\la}^-(y) + \phi_{\la}^+(y) \phi_{\la}^-(x)$$ and \eqref{fractional_basic_inequality}, we obtain 
\begin{align*}
   & \int_{\om^+} |\grad \phi_{\la}^+|^{p}  + \rnnn \frac{|\phi_{\la}^+(x) - \phi_{\la}^+(y)|^{p}}{|x - y|^{N+ps}} - \mu \int_{\om^+} \frac{|\phi_{\la}^+|^{p} }{|x|^{p \theta}}\notag \\
    &< \int_{\om^+} |\grad \phi_{\la}^+|^{p} + \rnnn \frac{F(\phi_{\la}(x) - \phi_{\la}(y))(\phi_{\la}^+(x) - \phi_{\la}^+(y))}{|x - y|^{N+ps}} - \mu \int_{\om^+} \frac{|\phi_{\la}^+|^{p} }{|x|^{p \theta}} 
    = \la \int_{\om^+} | \phi_{\la}^+|^p.
\end{align*}
Next, by definition of $\la_1(\om^+)$, we get 
\begin{align*}
    \la_1(\om^+) \int_{\om^+} |\grad \phi_{\la}^+|^{p}  \leq  \int_{\om^+} |\grad \phi_{\la}^+|^{p}  + \rnnn \frac{|\phi_{\la}^+(x) - \phi_{\la}^+(y)|^{p}}{|x - y|^{N+ps}} - \mu \int_{\om^+} \frac{|\phi_{\la}^+|^{p} }{|x|^{p \theta}} 
    < \la \int_{\om^+} |\grad \phi_{\la}^+|^{p},
\end{align*}
which implies $\la > \la_1(\om^+)$. Following the same reasoning, we also have $\la > \la_1(\om^-)$. Hence $\la > \max\{ \la_1(\om^+), \la_1(\om^-)\}$, which concludes our result.
\end{proof}

\begin{theorem}\label{second_eigenvalue_characterisation}
    Let $2 \leq p < \infty$ and $\om $ be a domain in $\rnn$. Then there exists a number $\la_2(\om)$ such that it satisfies:
    \begin{itemize}
        \item $\la_2$ is an eigenvalue of the operator $\mathcal{T}$
        \item $\la_1(\om) < \la_2(\om)$
        \item if $\la > \la_1(\om)$, then $\la \geq \la_2$.
    \end{itemize}
Moreover, $\la_2(\om)$ has the following characterization:
\begin{equation*}
    \la_2(\om) = \inf_{\gamma \in \Gamma} \max_{u \in \gamma} \mathcal{J}(u),
\end{equation*}
where $\mathcal{J}$ as in \eqref{energy_functional}, $\Gamma = \{ \gamma \in C([-1,1], S): \gamma(-1) = - \phi_1 \mbox{ and } \gamma(1) = \phi_1 \}$ with $\phi_1$ is the first eigenfunction of the operator $\mathcal{T}$.
\end{theorem}
\begin{proof}
Take
\begin{equation*}
    \la_* := \inf_{\gamma \in \Gamma} \max_{u \in \gamma[-1,1]} \mathcal{J}(u).
\end{equation*}
Then $\la^*$ is a critical value of functional $\mathcal{J}$, and this can be proved in the same way as $c(d)$ is proved in Theorem \ref{curve}. Clearly, $\la^* > \la_1(\om)$ as $\la_1(\om)$ is the strict minimum value of functional $\mathcal{J}$ achieved only at points $c \phi$ where $c \in \real \setminus \{0 \}$ and $\phi$ is the first eigenfunction.
Observing that $\mathcal{J}$ is even, arguing as in Theorem \ref{first_non_trivial_curve}, we obtain the required properties of $\la^*$.\\
Now, if $\gamma \in \Gamma$ is any curve, then take $\gamma_0 = \{ \gamma \} \cup \{- \gamma\}$. So, $\gamma_0$ is a set with $\gamma^*(\gamma_0) \geq 2$. Since, $\mathcal{J}$ is even, by characterization of second eigenvalue in Theorem \ref{existence_of_eigenvalues} we have $\la_2(\om) \leq \la^*$. But, $\la^*$ is the first eigenvalue after $\la_1(\om)$, so we obtain $\la_2(\om) = \la^*$.
\end{proof}

\begin{theorem}[Mixed Hong-Krahn-Szeg\"{o} inequality]
Let $p \geq 2$, $\om \subset \rnn$ be an open bounded set, and $B$ be any ball of volume $|\om|/2$. Then $$ \la_2(\om) > \la_1(B).$$
Additionally, if $\{ s_k \}$ and $\{ t_k \}$ are two sequences in $\rnn$ with $\displaystyle\lim_{k \to \infty}|s_k - t_k| = \infty$ and $\om_k = B_r(s_k) \cup B_r(t_k)$, then $\displaystyle \lim_{k \to \infty} \la_2(\om_k) = \la_1(B_r)$.
\end{theorem}
\begin{proof}
The proof follows similarly as the proof of Theorem $1.4$ in \cite{divya_second_eigenvalue_mixed} along with the genus characterization of $\la_2(\om)$ in Theorem \ref{second_eigenvalue_characterisation}.
The sketch of the proof goes like this:
We consider the second eigenfunction $\phi_2$ and define the corresponding positive and negative regions:
\begin{align*}
    \Omega^+ = \{ x \in \Omega \setminus \{ 0 \} : \phi_2(x) > 0 \}, \quad \Omega^- = \{ x \in \Omega \setminus \{ 0 \}: \phi_2(x) < 0 \}.
\end{align*}
Using Theorem \ref{thm:nodal_domains}, we know that
\begin{align*}
    \lambda_2(\Omega) > \max \{ \lambda_1(B_{r_1}), \lambda_1(B_{r_2}) \},
\end{align*}
where $B_{r_1}$ and $B_{r_2}$ are balls satisfying $|B_{r_1}| = |\Omega^+|$ and $|B_{r_2}| = |\Omega^-|$.

\noi Then one can show that $\max \{ \lambda_1(B_{r_1}), \lambda_1(B_{r_2}) \}$ is minimized when $|B_{r_1}| = |B_{r_2}| = |\Omega|/2$.
To show that equality can be achieved asymptotically. We define a sequence of domains as:
\begin{align*}
    \Omega_k := B_r(s_k) \cup B_r(t_k),
\end{align*}
where the centres $s_k$ and $t_k$ are chosen such that $|s_k - t_k| \to \infty$ as $k \to \infty$. Taking the normalized eigenfunctions $\phi_{s_k}$ and $\phi_{t_k}$ on these balls and construct a sequence of functions \( f_k: \mathbb{S}^1 \to S \) given by 
\[
f_k(\eta_1, \eta_2) = |\eta_1|^{\frac{2-p}{p}} \eta_1 \phi_{s_k} - |\eta_2|^{\frac{2-p}{p}} \eta_2 \phi_{t_k}.
\]
Then define \( A_k = \text{Range}(f_k) \). It implies that \( A_k \) is compact, symmetric, and of genus \( \geq 2 \). Using the variational characterization of $\lambda_2(\Omega)$ in Theorem \ref{existence_of_eigenvalues}, change of variable and the inequality $|a-b|^p\leq |a|^p+|b|^p+c_p(|a|^2+|b|^2)^{\frac{p-2}{2}}|ab|$
% \begin{proposition}
% Let $1<p<\infty$ and $a, b \in \mathbb{R}$, then there exists $c_p>0$ such that 
% \begin{equation}\label{eq:upper_p_bound}
% 	|a-b|^p\leq |a|^p+|b|^p+c_p(|a|^2+|b|^2)^{\frac{p-2}{2}}|ab|.
% \end{equation}
% \end{proposition}
% \eqref{eq:upper_p_bound} 
with $a = \Phi_{s_k}(x,y) (:= \phi_{s_k}(x)-\phi_{s_k}(y))$ and $b = \Phi_{t_k}(x,y)(:= \phi_{t_k}(x)-\phi_{t_k}(y))$, we obtain
 \begin{align*}
 &\la_2(\om_k) \leq \max_{|\eta_1|^2 + |\eta_2|^2 = 1} \mathcal{J}(f_k(\eta_1, \eta_2))\\ 
 & = \max_{|\xi_1|^p+|\xi_2|^p=1} \bigg \{ \int_{\om_k}|\na (\xi_1\phi_{s_k}-\xi_2\phi_{t_k})|^p+ \int_{\mathbb{R}^N}\int_{\mathbb{R}^N} \frac{|\xi_1 \Phi_{s_k}(x,y)- \xi_2 \Phi_{t_k}(x,y)|^p}{|x-y|^{N+ps}} \\
 % &  \hspace{4 cm}- \mu \int_{\om_k}\frac{|\xi_1\phi_{s_k}-\xi_2\phi_{t_k}|^p}{|x|^{p \theta}}\bigg\}\\
 % & = \max_{|\xi_1|^p+|\xi_2|^p=1} \bigg \{ \int_{\om_k}|\na \xi_1\phi_{s_k}|^p+ \int_{\om_k}|\na \xi_2\phi_{t_k}|^p+ \int_{\mathbb{R}^N}\int_{\mathbb{R}^N} \frac{|\xi_1 \Phi_{s_k}(x,y)- \xi_2 \Phi_{t_k}(x,y)|^p}{|x-y|^{N+ps}}~dxdy \\
  & \hspace{4 cm} -\mu \int_{\om_k}\frac{|\xi_1\phi_{s_k}|^p}{|x|^{p \theta}}~dx - \mu  \int_{\om_k}\frac{|\xi_2\phi_{t_k}|^p}{|x|^{p \theta}}~dx\bigg\}\\
 & \leq  \max_{|\xi_1|^p+|\xi_2|^p=1} \bigg \{ \int_{\om_k}|\na \xi_1\phi_{s_k}|^p~dx+ \int_{\om_k}|\na \xi_2\phi_{t_k}|^p ~dx + \int_{\mathbb{R}^N}\int_{\mathbb{R}^N} \frac{|\xi_1 \Phi_{s_k}(x,y)|^p}{|x-y|^{N+ps}}~dxdy\\
 &  \hspace{2.5cm} + \int_{\mathbb{R}^N}\int_{\mathbb{R}^N} \frac{|\xi_2 \Phi_{t_k}(x,y)|^p}{|x-y|^{N+ps}}~dxdy -\mu \int_{\om_k}\frac{|\xi_1\phi_{s_k}|^p}{|x|^{p \theta}}~dx - \mu  \int_{\om_k}\frac{|\xi_2\phi_{t_k}|^p}{|x|^{p \theta}}~dx\\
 &  \hspace{2.5cm} +  \int_{\mathbb{R}^N}\int_{\mathbb{R}^N} \frac{c_p(|\xi_1 \Phi_{s_k}(x,y)|^2+ |\xi_2 \Phi_{s_k}(x,y)|^2)^{\frac{p-2}{2}} |\xi_1\xi_2 \Phi_{s_k}(x,y) \Phi_{t_k}(x,y)|}{|x-y|^{N+ps}} 
 \bigg\}\\
 & \leq \la_1(B_r) + \max_{|\xi_1|^p+|\xi_2|^p=1} \int_{\mathbb{R}^N}\int_{\mathbb{R}^N}\frac{c_p(|\xi_1\Phi_{s_k}(x,y)|^2+ |\xi_2\Phi_{t_k}(x,y)|^2)^{\frac{p-2}{2}} |\xi_1\xi_2\Phi_{s_k}(x,y) \Phi_{t_k}(x,y)|}{|x-y|^{N+ps}}.
  \end{align*}
Since $\Phi_{s_k}(x,y) \Phi_{t_k}(x,y)= -\phi_{s_k}(x)\phi_{t_k}(y)- \phi_{s_k}(y)\phi_{t_k}(x)$ is nonzero only when $(x,y)\in B_r(s_k)\times B_r(t_k) \cup  B_r(t_k)\times B_r(s_k)$. And $|x-y| \geq |s_k - t_k| - 2r$ for all  $(x,y)\in B_r(s_k)\times B_r(t_k) \cup  B_r(t_k)\times B_r(s_k)$. Hence for sufficiently large $k$ such that $|s_k - t_k| - 2r > 0$, we have
\begin{align*}
	 \la_2(\om_k) & \leq  \la_1(B_r) +  \frac{2c_p}{(|s_k-t_k|-2r)^{N+ps}}\\
	 & \times\max_{|\xi_1|^p+|\xi_2|^p=1} \int_{B_r(s_k)}\int_{B_r(t_k)}(|\xi_1\Phi_{s_k}(x,y)|^2+ |\xi_2 \Phi_{t_k}(x,y)|^2)^{\frac{p-2}{2}} |\xi_1\xi_2 \Phi_{s_k}(x,y) \Phi_{t_k}(x,y)| .
\end{align*}
Since   \begin{align*}
	 2c_p\max_{|\xi_1|^p+|\xi_2|^p=1} \int_{B_r(s_k)}\int_{B_r(t_k)}(|\xi_1 \Phi_{s_k}(x,y)|^2+ |\xi_2 \Phi_{t_k}(x,y)|^2)^{\frac{p-2}{2}} |\xi_1\xi_2 \Phi_{s_k}(x,y) \Phi_{t_k}(x,y)| 
	\end{align*} is finite and $\frac{1}{(|s_k-t_k|-2r)^{N+ps}} \ra 0 $ as $k \ra \infty$. Thus $ \ds \lim_{k\ra \infty}\la_2(\om_k) \leq \la_1(B_r)$.

\end{proof}

\noindent{\bf Acknowledgment:} 
The first author is supported by the Prime Minister’s Research Fellowship(PMRF) with PMRF ID 1403226. The second author would like to thank the Science and Engineering Research Board, Department of Science and Technology, Government of India for the financial support under the grant SPG/2022/002068.\\

\bibliographystyle{plain}
\bibliography{references}
\end{document}